\documentclass[bj,MSNbibl,number,citesort,amsmath,v1.2,seceqn]{imsart}
\usepackage{amsthm,amsmath,amsfonts,amssymb}
\usepackage[numbers]{natbib}

\usepackage{mathrsfs}
\usepackage{enumerate}
\usepackage{tikz}
\usepackage{float}

\RequirePackage[colorlinks,citecolor=blue,urlcolor=blue]{hyperref}

\volume{00}
\issue{00}
\pubyear{0000}
\firstpage{1}
\lastpage{30}
\doi{10.3150/18-BEJ1072}% Updated by VTEXPTS2LaTeX.exe, 21.11.2018
\numberwithin{equation}{section}

\startlocaldefs
\newcommand{\rrvert}{\vert}
\newcommand{\rrVert}{\Vert}
\newcommand{\llvert}{\vert}
\newcommand{\llVert}{\Vert}
\newtheorem{theorem}{Theorem}[section]
\newtheorem{corollary}{Corollary}[section]
\newtheorem{definition}{Definition}[section]
\newtheorem{lemma}{Lemma}[section]
\newtheorem{proposition}{Proposition}[section]
\newtheorem{example}{Example}[section]
\newtheorem{remark}{Remark}[section]

\endlocaldefs

\begin{document}
\begin{frontmatter}

\title{A multivariate Berry--Esseen theorem with explicit constants}
\runtitle{A multivariate Berry--Esseen theorem}
%
%% if relates to special (Discussion,Rejoinder, etc.) article(s) make
%active and fill in the details:
%\relateddois{T1}{Discussed in \relateddoi{d}{?}%
%%; rejoinder at \relateddoi{r}{?}.
%}

\begin{aug}
%
%%% initials without spaces
%% if 2 or more \printead in one \address, then \ead[...,mark]
%
% {\inits{}\fnms{}~\snm{}\thanksref{}\ead[label=e?]{}
%\ead[label=u1,url]{}}
%
\author{\fnms{Martin}~\snm{Rai\v{c}}\ead
[label=e1]{martin.raic@fmf.uni-lj.si}\ead
[label=e2,url]{valjhun.fmf.uni-lj.si/\textasciitilde raicm/}}
%% Move support info to acks
%%\runauthor{} %% auto
\dedicated{This paper is dedicated to the memory of Vidmantas Kastytis
Bentkus (1949--2010).}
%
%%%%%% printead order as in MS:
\address{University of Ljubljana,
University of Primorska, and
Institute of Mathematics, Physics and Mechanics,
Slovenia. \printead{e1,e2}}
%%%%%%
\end{aug}

% HISTORY:
\received{\smonth{2} \syear{2018}}% Updated by VTEXPTS2LaTeX.exe,
%21.11.2018 14:01
\revised{\smonth{9} \syear{2018}}% Updated by VTEXPTS2LaTeX.exe,
%21.11.2018 14:01

% ABSTRACT
%
\begin{abstract}
We provide a Lyapunov type bound in the multivariate central limit
theorem for sums of independent, but not necessarily identically
distributed random vectors. The error in the normal approximation
is estimated for certain classes of sets, which include the class
of measurable convex sets. The error bound is stated with explicit
constants. The result is proved by means of Stein's method. In
addition, we improve the constant in the bound of the Gaussian
perimeter of convex sets.
\end{abstract}

% KEYWORDS
%
\end{frontmatter}
%
%\tableofcontents[level=?]%% if 50 pages and more AND used in MS (same
%level as in MS)
%%%%%%%%%%%%%%%%%%%%%%%%%%%%%%%%%%%%%%%%%%%%%%%%%%%%%%%%%%%%%%%%%%%%%%%%%%
%%%% Main text entry area:

\section{Introduction and results}
\label{sc:Intr}

Let $ \mathscr I $ be a countable set (either finite or infinite) and
let $ X_i $, $ i \in\mathscr I $, be independent $ \mathbb{R}^d $-valued
random vectors. Assume that $ \mathbb{E}X_i = 0 $ for all $ i $ and that
$ \sum_{i \in\mathscr I} \operatorname{Var}(X_i) = \mathbf{I}_d $.
It is well known
that in this case, the sum $ W := \sum_{i \in\mathscr I} X_i $
exists almost surely and that $ \mathbb{E}W = 0 $ and $ \operatorname
{Var}(W) = \mathbf{I}_d $.

For $ \mu\in\mathbb{R}^d $ and $ \boldsymbol{\Sigma} \in\mathbb
{R}^{d \times d} $, denote by
$ \mathcal{N}(\mu, \boldsymbol{\Sigma}) $ the $ d $-variate normal
distribution with
mean $ \mu$ and covariance matrix $ \boldsymbol{\Sigma} $. For a
measurable set
$ A \subseteq\mathbb{R}^d $, let $ \mathcal{N}(\mu, \boldsymbol{\Sigma})\{ A \} := \mathbb{P}(Z \in A) $,
and for a measurable function $ f \colon\mathbb{R}^d \to\mathbb
{R}$, denote
$ \mathcal{N}(\mu, \boldsymbol{\Sigma})\{ f \} := \mathbb{E} [
f(Z) ] $, where
$ Z \sim\mathcal{N}(\mu, \boldsymbol{\Sigma}) $.

Roughly speaking, the $ d $-variate central limit theorem for this
set-up says that if none of the summands $ X_i $ is ``too large'', the
sum $ W $ approximately follows $ \mathcal{N}(0, \mathbf{I}_d) $. The error
can be measured and estimated in various ways. Here, we focus on
the Lyapunov type bound
\begin{equation}
\label{eq:ErrBd} \sup_{A \in\mathscr A} \bigl\llvert \mathbb{P}(W \in A) -
\mathcal{N}(0, \mathbf{I}_d) \{ A \} \bigr\rrvert \le K \sum
_{i \in\mathscr I} \mathbb{E} \llvert X_i \rrvert
^3 ,
\end{equation}
where $ \mathscr A $ is a suitable class of subsets of $ \mathbb{R}^d
$ and
where $ \llvert x \rrvert $ denotes the Euclidean norm of the vector $ x $.

Fixing a class of sets for all dimensions $ d $, an important question
is the
dependence of the constant $ K $ on the dimension. The latter has drawn
the attention
of many authors and was tackled by different techniques. The class of measurable
convex sets appears as a natural extension of the classical univariate
Berry--Esseen theorem. For this case and for identically distributed
summands,
%.% Sazonov~\cite{SazVec2}: $ d^{5/2} $, Lindeberg--Bergstr\"om
Nagaev~\cite{Nag} uses Fourier transforms to derive a constant of
order $ d $.
%.% Sazonov~\cite{SazLNM}: the same by the method of compositions
%(Lindeberg--Bergstr\"om).
Bentkus~%
%FIXME: MR number only available for the Russian version
%\cite{BKoren:R:R}.
\cite{BKoren:E} succeeds to derive a constant of order $ d^{1/2} $ by
the method of composition (Lindeberg--Bergstr\"om method).
%.% Senatov~(1998), cited by Bentkus~\cite{BDim,BLja:E}, cites Bentkus~
%\cite{BKoren:R:R}.
Improving
this method and taking advantage of new bounds on Gaussian perimeters
of convex
sets (see below), he obtains $ K = 400 d^{1/4} $ in \cite{BDim}.
In
%FIXME: MR number only available for the Russian version \cite{BLja:R}.
\cite{BLja:E}, the latter result is extended to not necessarily
identically distributed
summands, but with no explicit constant, just of order $ d^{1/4} $.

In 1970, Stein~\cite{St0} developed a new elegant approach to bound
the error in
the normal approximation. His method was subsequently extended and
refined in many ways.
G\"otze~\cite{Gtz} derives \eqref{eq:ErrBd} with $ K = 157.85 d + 10
$ using Stein's method
combined with induction. Combining with part of Bentkus's argument,
Chen and Fang~\cite{ChFKoren}
succeed to improve this bound to $ 115 d^{1/2} $. However, this is
still of larger order
than Bentkus's result.

There used to be certain doubts about the correctness of G\"otze's
paper \cite{Gtz}.
To present a more readable account of G\"otze's paper, Bhattacharya and
Holmes wrote
an exposition \cite{BhHExp} of the arguments. However, they obtain a
higher order
dependence of the error rate on $ d $, namely $ d^{5/2} $. In
Remark~\ref{rk:BhHExpGain},
we explain where they gain the extra factor of $ d^{3/2} $.

Here, we combine G\"otze's and Bentkus's arguments to derive the
following explicit variant
of Bentkus's result:

\begin{theorem}
\label{th:ConvBd}
For $ X_i $ and $ W $ as above and all measurable convex sets $ A \subseteq\mathbb{R}^d $,
we have
\begin{equation}
\label{eq:ConvBd} \bigl\llvert \mathbb{P}(W \in A) - \mathcal{N}(0,
\mathbf{I}_d) \{ A \} \bigr\rrvert \le \bigl( 42 d^{1/4} + 16
\bigr) \sum_{i \in\mathscr I} \mathbb{E} \llvert X_i
\rrvert ^3 .
\end{equation}
\end{theorem}

This result follows immediately from Theorems~\ref{th:PerimBd} and
\ref{th:ClassBd} below,
also noticing the observations in Example~\ref{ex:ConvOK}.

To derive $ K $ in \eqref{eq:ErrBd}, it seems inevitable to include
\emph{Gaussian perimeters} of sets $ A \in\mathscr A $ or quantities closely
related to them. The Gaussian perimeter of a set $ A \subseteq\mathbb
{R}^d $ is defined
as
\[
\gamma(A) :=
\int_{\partial A} \phi_d(z) \mathscr H^{d-1}(
\mathrm {d}z) ,
\]
where $ \partial A $ denotes the topological boundary of $ A $, $
\mathscr H^{d-1} $ denotes
the $ (d-1) $-dimensional Hausdorff measure and $ \phi_d(z) := (2 \pi
)^{-d/2} \exp(- \llvert x \rrvert ^2/2) $
denotes the standard $ d $-variate Gaussian density.

Gaussian perimeters are closely related to Gaussian measures of
neighborhoods of the boundary.
Before stating it precisely, we introduce some notation:
\begin{itemize}
\item For a point $ x \in\mathbb{R}^d $ and a non-empty set $ A
\subseteq\mathbb{R}^d $, denote by $ \operatorname{dist}(x, A) $
the Euclidean distance from $ x $ to $ A $.
\item For a set $ A \subseteq\mathbb{R}^d $, which is neither the
empty set nor the whole $ \mathbb{R}^d $, define
the \emph{signed distance function} of $ A $ as
\[
% \label{eq:DistFunct}
\delta_A(x) := %
\begin{cases} -
\operatorname{dist}\bigl(x, \mathbb{R}^d \setminus A\bigr) ;& x
\in A,
\\
\operatorname{dist}(x, A) ;& x \notin A . \end{cases} %
\]
Moreover, for each $ t \in\mathbb{R}$, define $ A^t := \{ x \in
\mathbb{R}^d ;\delta_A(x) \le t \} $.
In addition, define $ \varnothing ^t := \varnothing $ and $ (\mathbb
{R}^d)^t := \mathbb{R}^d $.
\item For $ A \subseteq\mathbb{R}^d $, define
\[
\gamma^*(A) := \sup \biggl\{ \frac{1}{\varepsilon} \mathcal{N}(0,
\mathbf{I}_d)\bigl\{ A^\varepsilon \setminus A \bigr\},
\frac{1}{\varepsilon} \mathcal{N}(0, \mathbf{I}_d)\bigl\{ A \setminus
A^{- \varepsilon} \bigr\} ; \varepsilon> 0 \biggr\} .
\]
\item For a class of sets $ \mathscr A $, define
$ \gamma(\mathscr A) := \sup_{A \in\mathscr A} \gamma(A) $
and
$ \gamma^*(\mathscr A) := \sup_{A \in\mathscr A} \gamma^*(A) $.
\end{itemize}

The following proposition is believed by some authors to be evident.
However, though the
proof is quite straightforward, the assertion is not immediate. As a
special case of
Proposition~\ref{pr:PerimAnn}, it is proved in Section~\ref{sc:Perim}.

\begin{proposition}
\label{pr:PerimAnnNorm}
Let $ \mathscr A $ be a class of certain convex sets. Suppose that $
A^t \in\mathscr A \cup\{ \varnothing \} $
for all $ A \in\mathscr A $ and all $ t \in\mathbb{R}$. Then we have
$ \gamma(\mathscr A) = \gamma^*(\mathscr A) $.
\end{proposition}

Let $ \mathscr C_d $ be the class of \emph{all} convex sets in $
\mathbb{R}^d $. Denote
$ \gamma_d := \gamma(\mathscr C_d) = \gamma^*(\mathscr C_d) $. It is
known that
$ \gamma_d \le4 d^{1/4} $ -- see Ball~\cite{Bal}. Nazarov~\cite{Naz}
shows that the order $ d^{1/4} $ is correct and improved the upper bound
\emph{asymptotically}, showing that $ \limsup_{d \to\infty}
d^{-1/4} \gamma_d \le(2 \pi)^{1/4} < 0.64 $. Our next result
provides an explicit
bound, which is asymptotically even slightly better than Nazarov's bound.

\begin{theorem}
\label{th:PerimBd}
For all $ d \in\mathbb{N}$, we have
\begin{equation}
\label{eq:PerimBd} \gamma_d \le\sqrt{\frac{2}{\pi}} + 0.59 \bigl(
d^{1/4} - 1 \bigr) < 0.59 d^{1/4} + 0.21 .
\end{equation}
\end{theorem}

We defer the proof to Section~\ref{sc:Perim}.

\begin{remark}
Though $ \gamma_d $ is of order $ d^{1/4} $, this does not necessarily
mean that
this is the optimal order of the constant $ K $ in \eqref{eq:ErrBd}.
This remains
an open question.
\end{remark}

There are interesting classes of sets $ \mathscr A $ where there exist
better bounds on
$ \gamma(\mathscr A) $ than those of order $ d^{1/4} $. For the class
of all balls,
$ \gamma(\mathscr A) $ can be bounded independently of the dimension
-- see Sazonov~\cite{SazVec2,SazLNM}.
For the class of all rectangles, it is known that $ \gamma(\mathscr A)
$ is at most
of order $ \sqrt{\log d} $, see Nazarov~\cite{Naz}. Apart from convex
sets, other classes
may also be interesting, e.~g., the class of
% (possible infinite)
unions of balls which
are at least $ \Delta$ apart, where $ \Delta> 0 $ is a fixed number.
Therefore, we derive
a more general result; Theorem~\ref{th:ConvBd} will follow from the
latter and
Theorem~\ref{th:PerimBd}.

To generalize Theorem~\ref{th:ConvBd}, we shall consider a class $
\mathscr A $ of
measurable sets in $ \mathbb{R}^d $. For each $ A \in\mathscr A $,
take a measurable
function $ \rho_A \colon\mathbb{R}^d \to\mathbb{R}$. The latter can
be considered as a generalized signed distance function: typically, one
can take
$ \rho_A = \delta_A $, but we allow for more general functions. For each
$ t \in\mathbb{R}$, define
\[
A^{t \mid\rho} := \bigl\{ x ;\rho_A(x) \le t \bigr\}
% \quad\text{and, in addition,} \quad
% \emptyset^{t \mid\gdistf} := \emptyset
.
\]
Next, define the generalized Gaussian perimeter as
\begin{align*}
\gamma^*(A \mid\rho) &:= \sup \biggl\{ \frac{1}{\varepsilon} \mathcal{N}(0,
\mathbf{I}_d)\bigl\{ A^{\varepsilon
\mid\rho} \setminus A \bigr\},
\frac{1}{\varepsilon} \mathcal{N}(0, \mathbf{I}_d)\bigl\{ A \setminus
A^{- \varepsilon\mid\rho} \bigr\} ; \varepsilon> 0 \biggr\} ,
\\*
\gamma^*(\mathscr A \mid\rho) &:= \sup_{A \in\mathscr A} \gamma^*(A \mid
\rho) .
\end{align*}
We shall impose the following assumptions:
\def\labelenumi{(A\arabic{enumi})}
\begin{enumerate}[({A}1)]
\item\label{A:TrScale}\label{A:first}
$ \mathscr A $ is closed under translations and uniform scalings by
factors greater than one.
\item\label{A:At}
For each $ A \in\mathscr A $ and $ t \in\mathbb{R}$,
$ A^{t \mid\rho} \in\mathscr A \cup\{ \varnothing , \mathbb{R}^d
\} $.
\item\label{A:A-eps}
For each $ A \in\mathscr A $ and $ \varepsilon> 0 $, either $ A^{-
\varepsilon\mid\rho} = \varnothing $ or
$ \{ x ;\rho_{A^{- \varepsilon\mid\rho}}(x) < \varepsilon\}
\subseteq A $.
\item\label{A:delta0}
For each $ A \in\mathscr A $, $ \rho_A(x) \le0 $ for all $ x \in A $
and $ \rho_A(x) \ge0 $ for all $ x \notin A $.
\item\label{A:deltaTrans}
For each $ A \in \mathscr A $ and each $ y \in \mathbb{R}^d $, $ \rho_{A + y}(x + y) = \rho_A(x) $ for all $ x \in \mathbb{R}^d $.
\item\label{A:deltaq}
For each $ A \in\mathscr A $ and each $ q \ge1 $, $ |\rho_{q A}(q
x)| \le q |\rho_A(x)| $
for all $ x \in\mathbb{R}^d $.
\item\label{A:deltaNonExp}
For each $ A \in\mathscr A $, $ \rho_A $ is non-expansive on $ \{ x
;\rho_A(x) \ge0 \} $,
i.e., $ |\rho_A(x) - \rho_A(y)| \le|x - y| $ for all $ x, y $ with
$ \rho_A(x) \ge0 $
and $ \rho_A(y) \ge0 $.
\item\label{A:delta''}\label{A:last}
For each $ A \in\mathscr A $, $ \rho_A $ is differentiable on $ \{ x
;\rho_A(x) > 0 \} $.
Moreover, there exists $ \kappa\ge0 $, such that
\[
\bigl\llvert \nabla\rho_A(x) - \nabla\rho_A(y) \bigr
\rrvert \le\frac{\kappa \llvert x -y \rrvert }{\min\{ \rho_A(x), \rho_A(y) \}}
\]
for all $ x, y $ with $ \rho_A(x) > 0 $ and $ \rho_A(y) > 0 $;
throughout this paper, $ \nabla$
denotes the gradient.
\end{enumerate}
In addition, we state the following optional assumption:
\begin{enumerate}
\item[(A\ref{A:TrScale}$'$)] $ \mathscr A $ is closed under symmetric
linear transformations with the smallest
eigenvalue at least one.
\end{enumerate}

\def\labelenumi{(\arabic{enumi})}

\begin{remark}
It is natural to define $ \rho_A $ so that (A\ref{A:deltaq}) is
satisfied with
equality. However, for our main result, Theorem~\ref{th:ClassBd}, only
the inequality
is needed.
\end{remark}

\begin{remark}
Assumptions~(A\ref{A:A-eps})--(A\ref{A:delta''}) are hereditary: if
the pair
$ ( \mathscr A, (\rho_A)_{A \in\mathscr A} ) $
meets them and if $ \mathscr B \subseteq\mathscr A $, the pair
$ ( \mathscr B, (\rho_B)_{B \in\mathscr B} ) $
meets them, too.
\end{remark}

\begin{remark}
\label{rk:gdistfdistf}
With $ \rho_A = \delta_A $, one can easily check that
Assumptions~(A\ref{A:A-eps})--(A\ref{A:deltaNonExp}) are met.
Assumption~(A\ref{A:delta''}) is motivated
by Lemma~2.2 of Bentkus~\cite{BDim} (see Example~\ref{ex:ConvOK} below).
\end{remark}

The following is the main result of this paper.

\begin{theorem}
\label{th:ClassBd}
Let $ W = \sum_{i \in\mathscr I} X_i $ be as in Theorem~\ref{th:ConvBd}
and let $ \mathscr A $ be a class of sets meeting Assumptions~\textup{(A\ref
{A:first})--(A\ref{A:last})}
(along with the underlying functions $ \rho_A $). Then for each $ A
\in\mathscr A $, the
following estimate holds true:
\begin{equation}
\label{eq:ClassBd} \bigl\llvert \mathbb{P}(W \in A) - \mathcal{N}(0,
\mathbf{I}_d) \{ A \} \bigr\rrvert \le\max \bigl\{ 27, 1 + 53
\gamma^*(\mathscr A \mid\rho) \sqrt{1 + \kappa} \bigr\} \sum
_{i \in\mathscr I} \mathbb{E} \llvert X_i \rrvert
^3 .
\end{equation}
In addition, if $ \mathscr A $ also satisfies (A\ref{A:TrScale}$'$),
the preceding bound can be improved to
\begin{equation}
\label{eq:ClassBdAff} \bigl\llvert \mathbb{P}(W \in A) - \mathcal{N}(0,
\mathbf{I}_d) \{ A \} \bigr\rrvert \le\max \bigl\{ 27, 1 + 50
\gamma^*(\mathscr A \mid\rho) \sqrt{1 + \kappa} \bigr\} \sum
_{i \in\mathscr I} \mathbb{E} \llvert X_i \rrvert
^3 .
\end{equation}
\end{theorem}

We provide the proof in the next section.

\begin{remark}
Though explicit, the constants in Theorem~\ref{th:ClassBd} seem to be
far from optimal.
Consider the classical case where $ \mathscr A $ is the class of all
half-lines $ (- \infty, w] $,
where $ w $ runs over $ \mathbb{R}$. It is straightforward to check
that $ \mathscr A $ along with
$ \rho_A = \delta_A $ meets Assumptions~(A\ref{A:first})--(A\ref
{A:last}) with $ \kappa= 1 $.
Observing that $ \gamma^*(\mathscr A \mid\rho) = \gamma^*(\mathscr
A) = 1/\sqrt{2 \pi} $,
estimate~\eqref{eq:ClassBdAff} reduces
to \eqref{eq:ErrBd} with $ K = 29.3 $. This is much worse than $ K =
4.1 $
obtained by Chen and Shao~\cite{CS1} by Stein's method, let alone than
$ K = 0.5583 $ obtained by Shevtsova~\cite{Sevc0.56:E} by Fourier methods.
\end{remark}

Below we give further examples of classes of sets.

\begin{example}
\label{ex:ConvOK}
Consider the class $ \mathscr C_d $ of all measurable convex sets in $
\mathbb{R}^d $, along with
$ \rho_A = \delta_A $, which is defined in $ \mathscr C_d \setminus
\{ \varnothing , \mathbb{R}^d \} $.
Clearly, the latter class satisfies (A\ref{A:TrScale}). It is easy to
verify (A\ref{A:At}).
By Lemma~2.2 of Bentkus~\cite{BDim}, (A\ref{A:delta''}) is met with $
\kappa= 1 $.
By Remark~\ref{rk:gdistfdistf}, all other assumptions are met, too.
\end{example}

\begin{example}
The class of all balls in $ \mathbb{R}^d $ (excluding the empty set)
along with
$ \rho_A = \delta_A $ meets (A\ref{A:TrScale}) and (A\ref{A:At}).
Since the balls
are convex, it meets all Assumptions~(A\ref{A:first})--(A\ref{A:last}).
\end{example}

\begin{example}
For a class of ellipsoids, $ \rho_A = \delta_A $ is not suitable
because an
$ \varepsilon$-neighbor\-hood of an ellipsoid is not an ellipsoid.
However, one can set
$ \rho_A(x) := \delta_{\mathbf{Q} A}(\mathbf{Q} x) $, where $
\mathbf{Q} $
is a linear transformation mapping $ A $ into a ball (may depend on $ A $).
Notice that $ \mathbf{Q} $ must be non-expansive in order to satisfy
(A\ref{A:deltaNonExp}).
\end{example}

\begin{remark}
If the random vectors $ X_i $ are identically distributed, that is, if
$ \mathscr I $
has $ n $ elements and $ X_i $ follow the same distribution as $ \xi
/\sqrt{n} $,
the sum $ \sum_{i \in\mathscr I} \mathbb{E} \llvert X_i \rrvert ^3 $ reduces to $
n^{-1/2} \mathbb{E} \llvert \xi \rrvert ^3 $.
However, for the class of \emph{centered} balls, this rate of
convergence is
suboptimal. Using Fourier analysis, Esseen~\cite{Ess} succeeds to
derive a convergence
rate of order $ n^{-d/(d+1)} $ under the existence of the fourth
moment. This is possible
because of symmetry: that result is in fact an asymptotic expansion of first
order with vanishing first term.

Recently, Stein's method has been used by Gaunt, Pickett and
Reinert~\cite{GPRchi2} to
derive a convergence rate of order $ n^{-1} $, but for sufficiently
smooth radially
symmetric test functions rather than the indicators of centered balls. Applying
Stein's method to non-smooth test functions is not straightforward:
non-smoothness of
test functions needs to be compensated by a kind of smoothness of the
distribution of $ W $
or its modifications.

In the present paper, this is resolved by a `bootstrapping' argument
which is essentially
equivalent to G\"otze's~\cite{Gtz} inductive argument. The
probabilities of the sets
in the class $ \mathscr A $ are a kind of invariant (see \eqref
{eq:ClassBdAffKeta} and
\eqref{eq:ClassBdKeta}). In view of characteristic functions, this is
similar to the argument
introduced by Tihomirov~\cite{Tih}, which combines Stein's idea with
Fourier analysis.
Instead of the set probabilities, the invariant are the expectations of
functions
$ x \mapsto e^{\mathrm{i} \langle t , x \rangle} $ for $ t $ of
order $ O(\sqrt{n}) $. This
suffices to derive a convergence rate of order $ n^{-1/2} $.

Esseen~\cite{Ess} succeeds to go beyond this rate (in dimensions
higher than one)
by deriving a kind of smoothness of the distribution of $ W $ directly: see
Lemma~3 ibidem. This part of the argument seems to have no relationship
with Stein's method.
Similarly, Barbour and \v{C}ekanavi\v{c}ius~\cite{BrC} succeed to
sharply estimate the error
in the asymptotic expansions for integer random variables, but although
the main argument
is based on Stein's method, appropriate smoothness of modifications of
$ W $ is needed
and derived separately: see the inequality~(5.7) ibidem.

Unfortunately, smoothness of $ W $ in view of Lemma~3 of Esseen~\cite{Ess} is
unlikely to be useful in the argument used in this paper: another kind
of smoothness
would be desirable. Stein's method can be successfully combined with
the concentration
inequality approach, as in Chen and Fang~\cite{ChFKoren}. Certain modifications
of that approach could be a key to improvements.
\end{remark}

Now consider an example of a class of non-convex sets.

\begin{example}
Let $ \mathscr A $ be the class of all unions of disjoint intervals on
the real
line, such that the midpoints of any two intervals are at least $
\Delta$ apart,
where $ \Delta> 0 $ is fixed. In this case, $ \delta_A $ is not a suitable
function because it is not sufficiently smooth. We define $ \rho_A $ as
follows (see Figure~\ref{fi:gdistfNonConv}):
\begin{itemize}
\item If $ x \ge b = \sup A $, define $ \rho_A(x) := x - b $.
\item If $ x \le a = \inf A $, define $ \rho_A(x) := a - x $.
\item If $ x \notin A $ and $ b \le x \le a $, where $ b $ and $ a $
are the
endpoints of two successive intervals, define
\[
\rho_A(x) := \frac{1}{a - b} \biggl[ \biggl( \frac{a - b}{2}
\biggr)^2 - \biggl( x - \frac{b + a}{2} \biggr)^2
\biggr] .
\]
\item If $ x $ is an element of an interval with endpoints $ a $ and $
b $,
which constitutes $ A $, define
\[
\rho_A(x) := - \biggl( \frac{b - a}{2} - \biggl\llvert x -
\frac{a + b}{2} \biggr\rrvert \biggr) .
\]
\end{itemize}

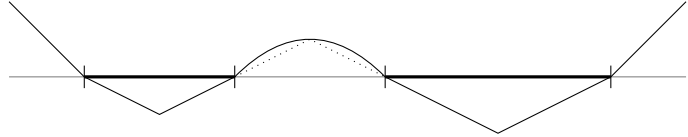
\begin{figure}[b]
\centering
\begin{tikzpicture}
\draw[gray] (-3, 0) -- (6, 0);
\draw[very thick] (-2, 0) -- (0, 0);
\draw[very thick] (2, 0) -- (5, 0);
\foreach\x in {-2, 0, 2, 5} { \draw(\x, -0.15) -- (\x, 0.15); }
\draw(-3, 1)
-- (-2, 0)
-- (-1, -0.5)
-- (0, 0)
.. controls (2/3, 2/3) and (4/3, 2/3) .. (2, 0)
-- (3.5, -0.75)
-- (5, 0)
-- (6, 1);
\draw[dotted] (0, 0) -- (1,0.5) -- (2,0);
\end{tikzpicture}
\caption{Construction of $ \rho_A $ for $ A = [-2, 0] \cup[2, 5] $.}
\label{fi:gdistfNonConv}
\end{figure}

Assumptions~(A\ref{A:TrScale}), (A\ref{A:At}) and (A\ref
{A:delta0})--(A\ref{A:deltaNonExp})
are easily verified (notice that some intervals may be joined or may
disappear under
$ A \mapsto A^{t \mid\rho} $, but the distances between their
midpoints never
decrease). To verify (A\ref{A:A-eps}), observe that for $ \rho_A(x)
\ge\delta_A(x)/2 $
for all $ x \notin A $. Consequently, $ A^{\varepsilon\mid\rho}
\subseteq A^{2 \varepsilon} $
for all $ \varepsilon> 0 $. Moreover, observe that $ A^{-\varepsilon
\mid\rho} = A^{-2 \varepsilon} $
for all $ \varepsilon> 0 $. As a result, either $ A^{- \varepsilon
\mid\rho} = \varnothing $ or
$ \{ x ;\rho_{A^{- \varepsilon\mid\rho}}(x) < \varepsilon\}
\subseteq
\{ x ;\delta_{A^{- 2 \varepsilon}}(x) < 2 \varepsilon\} \subseteq A $.

To verify (A\ref{A:delta''}), observe that if $ x \notin A $ and $ b
\le x \le a $, where
$ b $ and $ a $ are the endpoints of two successive intervals, we have
$ \rho''_A(x) \le2/(a - b) \le1/(2 \rho_A(x)) $. Thus, (A\ref
{A:delta''}) is
met with $ \kappa= 1/2 $.

Finally, we estimate $ \gamma^*(\mathscr A \mid\rho) $.
Let $ A \in\mathscr A $ be a union of disjoint intervals from $ a_j $
to $ b_j $, where
$ j $ runs over $ \mathscr J $, which is a set of successive numbers in
$ \mathbb{Z}$;
we can assume that the intervals appear in the same order as the
indices. Since
$ A^{\varepsilon\mid\rho} \subseteq A^{2 \varepsilon} $ and
$ A^{-\varepsilon\mid\rho} = A^{-2 \varepsilon} $, we have
$
\mathcal{N}(0, 1)(A^\varepsilon\setminus A)
\le\int_0^{2 \varepsilon} \sum_{j \in\mathscr J} [
\phi(a_j - t) + \phi(b_j + t)
] \,\mathrm{d}t
$
and
$
\mathcal{N}(0, 1)(A \setminus A^{- \varepsilon})
\le\int_{- 2 \varepsilon}^0 \sum_{j \in\mathscr J} [
\phi(a_j - t) + \phi( b_j + t)
] \,\mathrm{d}t
$,
where $ \phi$ denotes the standard univariate normal density, i.~e., $
\phi(x)
= (2 \pi)^{-1/2} e^{- x^2/2} $. Fix $ t $, consider the terms with $ a_j $ and $ b_j $ separately, and split the sums over the
indices where
$ a_j - t $ and $ b_j + t $ are positive or negative.
Estimating
$
a_{j+n} - a_j \ge\frac{a_{j+n-1} + b_{j+n-1}}{2} - \frac{a_j + b_j}{2}
\ge(n-1) \Delta
$
and
$
b_{j+n} - b_j \ge\frac{a_{j+n} + b_{j+n}}{2} - \frac{a_{j+1} + b_{j+1}}{2}
\ge(n-1) \Delta
$,
and applying
monotonicity of $ \phi$ on $ (- \infty, 0] $ and on $ [0, \infty) $,
we obtain
after some calculation
\begin{align*}
\frac{1}{\varepsilon} \max \bigl\{ \mathcal{N}(0, 1)\bigl\{ A^\varepsilon
\setminus A \bigr\}, \mathcal{N}(0, 1)\bigl\{A \setminus A^{- \varepsilon}\bigr
\} \bigr\} &\le\frac{8}{\sqrt{2 \pi}} \Biggl( 2 + \sum_{n=1}^\infty
e^{- n^2 \Delta^2/2} \Biggr)
\\
&\le\frac{8}{\sqrt{2 \pi}} \biggl( 2 +
\int_0^\infty e^{- \Delta^2 x^2/2} \,\mathrm{d}x \biggr)
\\
&= \frac{16}{\sqrt{2 \pi}} + \frac{4}\Delta .
\end{align*}
The latter is the desired upper bound on $ \gamma^*(\mathscr A \mid
\rho) $.
\end{example}

\section{Derivation of the bound in the central limit theorem}
\label{sc:CLT}

In this section, we prove Theorem~\ref{th:ClassBd}. We shall use the
ideas of
Bentkus~\cite{BDim} regarding smoothing and G\"otze~\cite{Gtz}
regarding Stein's
method. Before going to the proof,
we need a few auxiliary results; we defer their proofs to the end of
the section.
We also introduce some further notation and conventions.

Let $ x, u_1, u_2, \ldots, u_r \in\mathbb{R}^d $.
By $ \langle\nabla^r f(x) , u_1 \otimes u_2 \otimes\cdots
\otimes u_r \rangle $, we denote
the $ r $-th order derivative of $ f $ at $ x $ in directions $ u_1,
u_2, \ldots, u_r $.
By components, if $ u_i = (u_{i1}, u_{i2}, \ldots, u_{id}) $, we have
\[
\bigl\langle\nabla^r f(x) , u_1 \otimes
u_2 \otimes\cdots\otimes u_r \bigr\rangle = \sum
_{j_1, j_2, \ldots, j_r} \frac{\partial^r f(x)}{\partial x_{j_1} \partial x_{j_2} \cdots
\partial x_{j_r}} u_{1 j_1} u_{2 j_2}
\cdots u_{r j_r} .
\]
Thus, $ \nabla^r f(x) $ is a symmetric tensor of order $ r $. We
identify $ 2 $-tensors with
linear maps or their matrices by $ u \otimes v \equiv u v^T $. Observe
that the Laplace
operator can then be expressed as
\begin{equation}
\label{eq:Laplace} \Delta f(x) = \bigl\langle\nabla^2 f(x) ,
\mathbf{I}_d \bigr\rangle .
\end{equation}
By $ |T|_\vee$, we denote the \emph{injective norm} of tensor $ T $,
that is
\[
\llvert T \rrvert _\vee:= \sup_{|u_1|, |u_2|, \ldots, |u_r| \le1} \bigl\llvert
\langle T , u_1 \otimes u_2 \otimes\cdots\otimes
u_r \rangle \bigr\rrvert .
\]
For symmetric tensors, the supremum can be taken just over equal $ u_i $:

\begin{proposition}[Banach~\cite{Ban}; Bochnak and Siciak~\cite{BSi}]
\label{pr:InjSym}
If $ T $ is a symmetric tensor of order $ r $, then $ |T|_\vee= \sup_{|u| \le1} |\langle T , u^{\otimes r} \rangle| $.
\end{proposition}

Next, denote
\begin{align*}
M_0^*(f) &:= \frac{1}2 \Bigl[ \sup_{w \in\mathbb{R}^d}
f(w) - \inf_{w \in\mathbb{R}^d} f(w) \Bigr] ,
\\
M_r(f) &:= \sup_{\substack{w, z \in\mathbb{R}^d \\ w \ne z}} \frac{ \llvert \nabla^{r-1} f(w) - \nabla^{r-1} f(z) \rrvert _\vee
}{|w - z|} ;\qquad
r = 1, 2, 3, \ldots
\end{align*}
If $ f $ is not everywhere $ (r-1) $-times differentiable, we put $
M_r(f) = \infty$.

\begin{remark}
\label{rk:Rademacher}
This way, if $ M_r(f) < \infty$, then $ \nabla^{r-1} f $ exists
everywhere and is Lip\-schitzian.
In this case, by Rademacher's theorem (see Federer~\cite{Fed}, Theorem~3.1.6),
% Section~3.1.2 of Evans and Gariepy~\cite{EvG}
$ \nabla^{r-1} f $ is almost everywhere differentiable. In addition,
$ M_r(f) = \sup_x \llvert \nabla^r f(x) \rrvert _\vee$,
where the supremum runs over all points where $ \nabla^{r-1} f $ is
differentiable.
\end{remark}

Now we turn to auxiliary results regarding smoothing.
The following one is a counterpart of Lemma~2.3 of Bentkus~\cite{BDim}.

\begin{lemma}
\label{lm:Smoothdelta}
Let $ \mathscr A $ be a class of sets which, along with the underlying
functions $ \rho_A $,
meets Assumptions~\textup{(A\ref{A:first})--(A\ref{A:last})}. Then for each $
A \in\mathscr A $ and
each $ \varepsilon> 0 $, there exist functions $ f_A^\varepsilon,
f_A^{- \varepsilon} \colon\mathbb{R}^d \to\mathbb{R}$,
such that:

\begin{enumerate}[(1)]
\item\label{Smoothdelta01}
$ 0 \le f_A^\varepsilon, f_A^{- \varepsilon} \le1 $.
\item\label{Smoothdelta01+}
$ f_A^\varepsilon(x) = 1 $ for all $ x \in A $ and $ f_A^\varepsilon
(x) = 0 $ for all
$ x \in\mathbb{R}^d \setminus A^{\varepsilon\mid\rho} $.
\item\label{Smoothdelta01-}
$ f_A^{-\varepsilon}(x) = 1 $ for all $ x \in A^{- \varepsilon\mid
\rho} $ and
$ f_A^{-\varepsilon}(x) = 0 $ for all $ x \in\mathbb{R}^d \setminus
A $.
\item\label{SmoothdeltaM}
The following bounds hold true:
\[
% \label{eq:SmoothdeltaM}
M_1\bigl(f_A^\varepsilon\bigr) \le
\frac{2}{\varepsilon}, \quad M_1\bigl(f_A^{- \varepsilon}
\bigr) \le\frac{2}{\varepsilon}, \quad M_2\bigl(f_A^\varepsilon
\bigr) \le\frac{4(1 + \kappa)}{\varepsilon^2}, \quad M_2\bigl(f_A^{- \varepsilon}
\bigr) \le\frac{4(1 + \kappa)}{\varepsilon^2}.
\]
\item\label{SmoothdeltaInt}
For each $ u \in(0, 1) $, $ \{ x ;f_A^\varepsilon(x) \ge u \} \in
\mathscr A \cup\{ \varnothing , \mathbb{R}^d \} $
and $ \{ x ;f_A^{-\varepsilon}(x) \ge u \} \in\mathscr A \cup\{
\varnothing , \mathbb{R}^d \} $.
\end{enumerate}
\end{lemma}

\begin{proof}
First, define $ f_A^\varepsilon(x) := g ( \frac{\rho
_A(x)}{\varepsilon} ) $, where
\[
g(x) := %
\begin{cases} 1 ;& x \le0 ,
\\
1 - 2 x^2 ;& 0 \le x \le1/2 ,
\\
2 (1-x)^2 ;& 1/2 \le x \le1 ,
\\
0 ;& x \ge1 . \end{cases} %
\]
Requirements~(\ref{Smoothdelta01}) in (\ref{Smoothdelta01+}) are
immediate, while
(\ref{Smoothdelta01-}) is irrelevant for $ f_A^\varepsilon$. To prove
(\ref{SmoothdeltaInt}),
observe that $ \{ x ;f_A^\varepsilon(x) \ge u \} = A^{\varepsilon
g^{-1}(u) \mid\rho} $.
Now we turn to (\ref{SmoothdeltaM}). First, notice that $
M_1(f_A^\varepsilon) \le2/\varepsilon$ because
$ M_1(\rho_A) \le1 $ in $ M_1(g) = 2 $.
Next, $ f_A^\varepsilon$ is continuously differentiable: see
supplementary material \cite{BEJ1072-supp}.
Letting $ B := \{ x ; \rho_A(x) \le 0 \} $, take $ x, y \in\mathbb{R}^d \setminus B $ with $ \rho_A(x) \ge
\rho_A(y) $ and estimate
\begin{align*}
\bigl\llvert \nabla f_A^\varepsilon(x) - \nabla
f_A^\varepsilon(y) \bigr\rrvert\le{}& \frac{1}{\varepsilon} \biggl
\llvert g' \biggl( \frac{\rho_A(x)}{\varepsilon} \biggr) - g'
\biggl( \frac
{\rho_A(y)}{\varepsilon} \biggr) \biggr\rrvert \bigl\llvert \nabla
\rho_A(x) \bigr\rrvert
\\
& {} + \frac{1}{\varepsilon} \biggl\llvert g' \biggl(
\frac{\rho
_A(y)}{\varepsilon} \biggr) \biggr\rrvert \bigl\llvert \nabla\rho_A(x)
- \nabla\rho_A(y) \bigr\rrvert .
\end{align*}
In the first term, we apply $ M_2(g) = 4 $ and $ M_1(\rho_A) \le1 $,
while in the second, we apply $ |g'(t)| \le4 t $ and (A\ref{A:delta''}).
Combining these estimates, we obtain $ |\nabla f_A^\varepsilon(x) -
\nabla f_A^\varepsilon(y)|
\le4(1 + \kappa) \varepsilon^{-2} |x - y| $, noticing that we may
drop the
assumption that $ \rho_A(x) \ge\rho_A(y) $. In other words,
on $ \mathbb{R}^d \setminus B $, $ \nabla f_A^\varepsilon$ is
Lipschitzian with constant
$ 4(1 + \kappa) \varepsilon^{-2} $. Trivially, this also holds true
in the
interior of $ B $. Since $ \nabla f_A^\varepsilon$ is continuous, this also
holds true on the closures of both sets. Since for each $ x \in
\overline{\operatorname{Int}B} $
and each $ y \in\overline{\mathbb{R}^d \setminus B} $, there exists
$ z $ on the line
segment with endpoints $ x $ and $ y $, which is an element of both sets,
$ \nabla f_A^\varepsilon$ is Lipschitzian with the above-mentioned
constant on the
whole $ \mathbb{R}^d $. Thus, $ f_A^\varepsilon$ meets all relevant
requirements.

Now define $ f_A^{- \varepsilon} \equiv0 $ if $ A^{- \varepsilon} =
\varnothing $ and
$ f_A^{- \varepsilon} := f_{A^{- \varepsilon}}^\varepsilon$
otherwise. From the above and from
Assumption~(A\ref{A:A-eps}), it follows that this function also
satisfies all relevant
requirements. This completes the proof.
\end{proof}

Throughout this section, $ \boldsymbol{\Sigma} $ will refer to a
positive-definite
matrix $ \boldsymbol{\Sigma} $ with the largest eigenvalue at most one
and with the smallest eigenvalue $ \sigma^2 $, where $ \sigma> 0 $.

\begin{lemma}
\label{lm:PerimNonStd}
Let $ \mathscr A $ be a class of sets, which, along with the underlying
functions
$ \rho_A $, meets Assumptions~(A\ref{A:TrScale}) and (A\ref
{A:deltaq}). Then the following estimates hold true for all $ \varepsilon > 0 $:
\[
\mathcal{N}(\mu, \boldsymbol{\Sigma}) \bigl\{ A^{\varepsilon\mid\rho} \setminus A \bigr
\} \le\frac{\gamma^*(\mathscr{A} \mid\rho) \varepsilon}{\sigma} \quad\text{and} \quad \mathcal{N}(\mu, \boldsymbol{
\Sigma}) \bigl\{ A \setminus A^{- \varepsilon
\mid\rho} \bigr\} \le\frac{\gamma^*(\mathscr{A} \mid\rho) \varepsilon}{\sigma} .
\]
\end{lemma}

\begin{proof}
Take independent random vectors $ Z \sim\mathcal{N}(0, \mathbf{I}_d)
$ and
$ R \sim\mathcal{N}(\mu, \boldsymbol{\Sigma} - \sigma^2 \mathbf
{I}_d) $. Clearly,
$ \sigma Z + R \sim\mathcal{N}(\mu, \boldsymbol{\Sigma}) $. Now
observe that, by (A\ref{A:deltaTrans}),
\begin{align*}
\mathcal{N}(\mu, \boldsymbol{\Sigma}) \bigl\{ A^{\varepsilon\mid\rho} \setminus A \bigr
\} &= \mathbb{P} \bigl( \sigma Z + R \in A^{\varepsilon\mid\rho} \setminus A \bigr)
\\
&= \mathbb{P} \bigl[ Z \in\sigma^{-1} (A^{\varepsilon\mid
\rho} - R) \bigm
\backslash\sigma^{-1} (A - R) \bigr]
\\
&= \mathbb{P} \bigl[ Z \in\sigma^{-1} \bigl( (A - R)^{\varepsilon\mid
\rho} \bigr) \bigm
\backslash\sigma^{-1} (A - R) \bigr]
\\
&= \mathbb{E} \left[
\mathcal{N}(0, \mathbf{I}_d) \bigl\{ \sigma^{-1} \bigl( (A -
R)^{\varepsilon\mid\rho} \bigr) \bigm\backslash\sigma^{-1} (A - R) \bigr\} \right]
.
\end{align*}
From (A\ref{A:deltaq}), it follows that
$ \sigma^{-1} B^{\varepsilon\mid\rho} \subseteq(\sigma^{-1}
B)^{(\varepsilon/\sigma) \mid\rho} $
for all $ B \in\mathscr{A} $. Therefore,
\[
\mathcal{N}(\mu, \boldsymbol{\Sigma}) \bigl\{ A^{\varepsilon\mid\rho} \setminus A \bigr
\} \le \mathbb{E} \left[
\mathcal{N}(0, \mathbf{I}_d) \bigl\{ \bigl(
\sigma^{-1} (A - R) \bigr)^{(\varepsilon/\sigma) \mid\rho} \bigm\backslash \bigl(
\sigma^{-1} (A - R) \bigr) \bigr\} \right]
\le \frac{\gamma^*(\mathscr{A} \mid\rho) \varepsilon}{\sigma}
\]
(notice that $ \sigma^{-1}(A - R) \in\mathscr A $ by (A\ref{A:TrScale})).
Analogously, we obtain $ \mathcal{N}(\mu, \boldsymbol{\Sigma}) \{ A
\setminus A^{- \varepsilon} \}
\le\gamma^*(\mathscr{A} \mid\rho) \varepsilon/\sigma$. This
completes the proof.
\end{proof}

\begin{lemma}
\label{lm:ClassTrans}
Let a class $ \mathscr A $ along with the underlying functions $ \rho
_A $ meet Assumptions
\textup{(A\ref{A:first})--(A\ref{A:last})}. Take a linear map $ \mathbf{L}
\colon\mathbb{R}^d \to\mathbb{R}^d $
with the smallest singular value at least one. Then the class $ \tilde
{\mathscr A} :=
\{ \mathbf{L} A ;A \in\mathscr A \} $ along with the underlying functions
$ \tilde\rho_{\tilde A}(x) := \rho_{\mathbf{L}^{-1} \tilde
A}(\mathbf{L}^{-1} x) $
meets these assumptions with the same $ \kappa$ in \textup{(A\ref
{A:delta''})}. Moreover,
\begin{equation}
\label{eq:ClassTransPerim} \gamma^*(\tilde{\mathscr A} \mid\tilde\rho) \le \llVert
\mathbf{L} \rrVert \gamma^*(\mathscr A \mid\rho) .
\end{equation}
\end{lemma}

\begin{proof}
Assumptions~(A\ref{A:TrScale}), (A\ref{A:A-eps}), (A\ref{A:delta0}), (A\ref{A:deltaTrans})
and (A\ref{A:deltaq})
are straightforward to check. To verify~(A\ref{A:At}), observe that
\[
\tilde{A}^{t \mid\tilde\rho} = \bigl\{ x ;\tilde\rho_{\tilde A}(x) \le t \bigr
\} = \bigl\{ x ;\rho_{\mathbf{L}^{-1} \tilde{A}}\bigl(\mathbf{L}^{-1} x\bigr) \le t
\bigr\} = \mathbf{L} \bigl( \mathbf{L}^{-1} \tilde{A}
\bigr)^{t \mid
\rho} .
\]
Assumption~(A\ref{A:deltaNonExp}) follows from the fact that $
\mathbf{L}^{-1} $ is
non-expansive. To verify (A\ref{A:delta''}), observe that, by the
chain rule,
$ \nabla\tilde{\rho}_{\tilde A}(x) = \mathbf{L}^{-T}
\nabla\rho_{\mathbf{L}^{-1} \tilde A}(\mathbf{L}^{-1} x) $,
and use again that $ \mathbf{L}^{-1} $ is non-expansive.
Finally, observe that
\begin{align*}
\mathcal{N}(0, \mathbf{I}_d) \bigl\{ \tilde{A}^{\varepsilon\mid
\tilde{\rho}} \bigm
\backslash\tilde{A} \bigr\} &= \mathcal{N}(0, \mathbf{I}_d) \bigl\{
\mathbf{L} \bigl( \mathbf{L}^{-1} \tilde{A} \bigr)^{\varepsilon\mid\rho} \bigm
\backslash\tilde{A} \bigr\}
\\
&= \mathcal{N} \bigl( 0, \mathbf{L}^{-1} \mathbf{L}^{-T}
\bigr) \bigl\{ \bigl( \mathbf{L}^{-1} \tilde{A} \bigr)^{\varepsilon\mid\rho}
\bigm\backslash\mathbf{L}^{-1} \tilde{A} \bigr\}
\\
&\le \llVert \mathbf{L} \rrVert \gamma^*(\mathscr{A} \mid\rho) \varepsilon
\end{align*}
by Lemma~\ref{lm:PerimNonStd}. An analogous inequality holds true for
$ \tilde{A} \setminus\tilde{A}^{-\varepsilon\mid\tilde{\rho}} $.
Taking the supremum over $ \tilde A \in\tilde{\mathscr{A}} $, we obtain
\eqref{eq:ClassTransPerim}.
\end{proof}

Now we turn to Stein's method, which will be implemented in view of the proof
of Lemma~1 of Slepian~\cite{Slep}.
% , which is known as Slepian's lemma.
We recall
the procedure briefly; for an exposition, see R\"ollin~\cite{RolDim} and
Appendix~H of Chernozhukov, Chetverikov and Kato~\cite{CCK13sup}. Let
$ f $ be a
bounded measurable function. For $ 0 \le\alpha\le\pi/2 $, define
\begin{equation}
\label{eq:Ualpha} \mathscr U_\alpha f(w) :=
\int_{\mathbb{R}^d} f(w \cos\alpha+ z \sin\alpha) \phi_d(z)
\,\mathrm{d}z .
\end{equation}
For a random variable $ W $, $ \mathbb{E} [ \mathscr U_\alpha
f(W) ] $ can be regarded as an
interpolant between $ \mathbb{E} [ f(W) ] $ and $ \mathcal
{N}(0, \mathbf{I}_d)\{ f \} $. A straightforward
calculation shows that
\[
\frac{\mathrm{d}}{\mathrm{d}\alpha} \mathscr U_\alpha f(w) = \mathscr S \mathscr
U_\alpha f(w) \tan\alpha ,
\]
where $ \mathscr S $ denotes the \emph{Stein operator}:
\begin{equation}
\label{eq:SteinOp} \mathscr S g(w) := \Delta g(w) - \bigl\langle\nabla g(w) ,
w \bigr\rangle
\end{equation}
and where $ \Delta$ denotes the Laplacian. Integrating over $ \alpha$ and
taking expectation, we find that
\begin{equation}
\label{eq:SteinSlep} \mathbb{E}f(W) - \mathcal{N}(0, \mathbf{I}_d)\{ f \}
= -
\int _0^{\pi/2} \mathbb{E} \bigl[ \mathscr S \mathscr
U_\alpha f(W) \bigr] \tan\alpha \,\mathrm{d}\alpha .
\end{equation}
Notice that for $ 0 < \alpha\le\pi/2 $, $ \mathscr U_\alpha f $ is
infinitely differentiable,
so that $ \mathscr S \mathscr U_\alpha f $ is well-defined.
Differentiability can be
shown by integration by parts. In particular, we shall need
\begin{align}
\label{eq:f0phi3} \nabla^3 \mathscr U_\alpha f(w) &= -
\cot^3 \alpha
\int_{\mathbb{R}^d} f(w \cos\alpha+ z \sin\alpha) \nabla^3
\phi_d(z) \,\mathrm{d}z
\\
\label{eq:f2phi1} &= - \frac{\cos^3 \alpha}{\sin\alpha}
\int_{\mathbb{R}^d} \nabla^2 f(w \cos\alpha+ z \sin\alpha)
\otimes\nabla\phi_d(z) \,\mathrm{d}z .
\end{align}
The proof is straightforward and is therefore left to the reader (cf.\
Section~2 of
Bhattacharya and Holmes~\cite{BhHExp}). Observe that \eqref
{eq:f2phi1} remains true
for all $ w $ if $ \nabla f $ is Lipschitzian, that is, $ M_2(f) <
\infty$
(see Remark~\ref{rk:Rademacher}).

Now we turn to the \emph{Stein expectation} $ \mathbb{E} [
\mathscr S g(W) ] $.
The following result, which is essentially a counterpart of Lemma~2.9
of G\"otze~\cite{Gtz},
expresses it in a way which is useful for its estimation.

\begin{lemma}[Stein Expectation]\label{lm:Stein}
Let $ X_i $, $ i \in\mathscr I $, be independent $ \mathbb{R}^d
$-valued random vectors with sum
$ W $, which satisfies $ \mathbb{E}W = 0 $ and $ \operatorname
{Var}(W) = \mathbf{I}_d $. Then for any bounded
three times continuously differentiable function $ g $ with bounded derivatives,
\[
\mathbb{E} \bigl[ \mathscr S g(W) \bigr] = \sum_{i \in\mathscr I}
\mathbb{E} \bigl[ \bigl\langle\nabla^3 g(W_i + \theta
X_i) , X_i \otimes\tilde X_i^{\otimes2}
- (1 - \theta) X_i^{\otimes3} \bigr\rangle \bigr] ,
\]
where $ W_i = W - X_i $, $ \tilde X_i $ is an independent copy of $ X_i $,
$ \theta$ is uniformly distributed over $ [0, 1] $, and $ \tilde X_i $
and $ \theta$ are
independent of each other and all other variates.
\end{lemma}

\begin{proof}
Recalling~\eqref{eq:Laplace}, write
\begin{align*}
\Delta g(W) &= \bigl\langle\nabla^2 g(W) ,
\mathbf{I}_d \bigr\rangle = \bigl\langle\nabla^2 g(W) ,
\operatorname{Var}(W) \bigr\rangle = \sum_{i \in\mathscr I}
\bigl\langle\nabla^2 g(W) , \operatorname {Var}(X_i)
\bigr\rangle
\\
&= \sum_{i \in\mathscr I} \bigl\langle\nabla^2 g(W)
, \mathbb {E} \bigl( X_i^{\otimes2} \bigr) \bigr\rangle .
\end{align*}
Plugging into \eqref{eq:SteinOp}, we obtain
\begin{align*}
\mathbb{E} \bigl[ \mathscr S g(W) \bigr] &= \sum_{i \in\mathscr I}
\mathbb{E} \bigl[ \bigl\langle\nabla^2 g(W_i +
X_i) , \mathbb {E}X_i^{\otimes2} \bigr\rangle
- \bigl\langle\nabla g(W_i + X_i) ,
X_i \bigr\rangle \bigr]
\\
&= \sum_{i \in\mathscr I} \mathbb{E} \bigl[ \bigl\langle
\nabla^2 g(W_i + X_i) , \tilde
X_i^{\otimes2} \bigr\rangle - \bigl\langle\nabla
g(W_i + X_i) , X_i \bigr\rangle
\bigr] .
\end{align*}
Taylor expansion centered at $ W_i $ yields
\begin{align*}
\mathbb{E} \bigl[ \mathscr S g(W) \bigr] ={}& \sum_{i \in\mathscr I}
\mathbb{E} \bigl[ \bigl\langle\nabla^2 g(W_i) ,
\tilde X_i^{\otimes2} \bigr\rangle + \bigl\langle
\nabla^3 g(W_i + \theta X_i) ,
X_i \otimes\tilde X_i^{\otimes2} \bigr\rangle
\\
& {} - \bigl\langle\nabla g(W_i) ,
X_i \bigr\rangle - \bigl\langle\nabla^2
g(W_i) , X_i^{\otimes2} \bigr\rangle - (1 -
\theta) \bigl\langle\nabla^3 g(W_i + \theta
X_i) , X_i^{\otimes3} \bigr\rangle \bigr] .
\end{align*}
By independence, the first and the fourth term cancel and the third
term vanishes because\break
$ \mathbb{E}X_i = 0 $. This completes the proof.
\end{proof}

Now we turn to the estimation of several integrals related to the
multivariate normal distribution.
Define constants $ c_0, c_1, c_2, \ldots$ by
\[
c_r :=
\int_{- \infty}^\infty \bigl\llvert \phi^{(r)}_1(z)
\bigr\rrvert \,\mathrm{d}z .
\]

\begin{lemma}
\label{lm:IntDerphidBd}
For each bounded measurable function $ f $, each $ r \in\mathbb{N}$
and each $ u \in\mathbb{R}^d $,
we have
\[
\biggl\llvert
\int_{\mathbb{R}^d} f(z) \bigl\langle\nabla^r
\phi_d(z) , u^{\otimes r} \bigr\rangle \,\mathrm{d}z \biggr
\rrvert \le c_r M_0^*(f) \llvert u \rrvert ^r
.
\]
\end{lemma}

\begin{proof}
First, observe that since the function $ F(x) = \int_{\mathbb{R}^d}
\phi_d(z + x) \,\mathrm{d}z $
is constant, we have
$ \int_{\mathbb{R}^d} \langle\nabla^r \phi_d(z) , u^{\otimes
r} \rangle \,\mathrm{d}z
= \langle\nabla^r F(0) , u^{\otimes r} \rangle = 0 $.
Therefore, $ f $ can be replaced by $ f - b $, where $ b $ is arbitrary
constant.
As a result,
\[
\biggl\llvert
\int_{\mathbb{R}^d} f(z) \bigl\langle\nabla^r
\phi_d(z) , u^{\otimes r} \bigr\rangle \,\mathrm{d}z \biggr
\rrvert \le \sup \llvert f - b \rrvert
\int_{\mathbb{R}^d} \bigl\llvert \bigl\langle\nabla^r
\phi_d(z) , u^{\otimes r} \bigr\rangle \bigr\rrvert
\,\mathrm{d}z .
\]
Choosing $ b = (\inf f + \sup f)/2 $, we have $ \sup|f - b| = M_0^*(f)
$. Next,
since $ \phi_d $ is spherically symmetric, we can replace $ u $ by $
|u| e_1 $, where
$ e_1 = (1, 0, \ldots, 0) $. Writing $ z = (z_1, z') $, we have $
\langle\nabla^r \phi_d(z) , e_1^{\otimes r} \rangle
= \phi_1^{(r)}(z_1) \phi_{d-1}(z') $, so that
\[
\int_{\mathbb{R}^d} \bigl\llvert \bigl\langle\nabla^r
\phi_d(z) , e_1^{\otimes r} \bigr\rangle \bigr
\rrvert \,\mathrm{d}z =
\int_{\mathbb{R}} \bigl\llvert \phi_1^{(r)}(z_1)
\bigr\rrvert \,\mathrm{d}z_1
\int_{\mathbb{R}^{d-1}} \phi_{d-1}\bigl(z'\bigr)
\,\mathrm{d}z' = c_r .
\]
Combining this with previous observations, the result follows.
\end{proof}

\begin{remark}
\label{rk:BhHExpGain}
At this step, Bhattacharya and Holmes~\cite{BhHExp} gain the extra
factor of $ d^{3/2} $
in their bound. Instead of taking advantage of spherical symmetry, they
estimate by
components -- see the estimates~(3.12)--(3.15) ibidem. G\"otze's
paper~\cite{Gtz} comes to
this step in the estimate~(2.7) ibidem, where the result of Lemma~\ref
{lm:IntDerphidBd} is
actually used, but no argument is provided.
% G\"otze~\cite{Gtz} comes to this
% step in his estimate~(2.7), where he actually uses the result of
%Lemma~\ref{lm:IntDerphidBd},
% but does not provide an argument.
\end{remark}

\begin{lemma}
\label{lm:OUDerNorm}
Let $ f \colon\mathbb{R}^d \to\mathbb{R}$ be bounded and
measurable. Take
$ 0 < \alpha\le\pi/2 $. Then for all $ r \in\mathbb{N}$ and all $
\mu, u \in\mathbb{R}^d $,
\[
\bigl\llvert \bigl\langle\mathcal{N}(\mu, \boldsymbol{\Sigma}) \bigl\{
\nabla^r \mathscr U_\alpha f \bigr\} , u^{\otimes r}
\bigr\rangle \bigr\rrvert \le c_r M_0^*(f)
\frac{\cos^r \alpha}{\sigma^r} \llvert u \rrvert ^r .
\]
\end{lemma}

\begin{remark}
\label{rk:ETensor}
The expression $ \mathcal{N}(\mu, \boldsymbol{\Sigma}) \{
\nabla^r \mathscr U_\alpha f \} $
is an expectation of a random tensor of order $ r $ and is therefore a
deterministic tensor.
This allows us to define
\newline
$ \langle\mathcal{N}(\mu, \boldsymbol{\Sigma}) \{ \nabla^r
\mathscr U_\alpha f \} , u^{\otimes r} \rangle $.
\end{remark}

\begin{proof}[Proof of Lemma~\ref{lm:OUDerNorm}]
Write
\begin{equation}
\label{eq:OUDerNorm:DerF} \bigl\langle\mathcal{N}(\mu, \boldsymbol{\Sigma}) \bigl\{
\nabla^r \mathscr U_\alpha f \bigr\} , u^{\otimes r}
\bigr\rangle = \mathbb{E} \bigl[ \bigl\langle\nabla^r \mathscr
U_\alpha f\bigl(\boldsymbol{\Sigma}^{1/2} Z + \mu\bigr) ,
u^{\otimes r} \bigr\rangle \bigr] = \bigl\langle\nabla^r F(\mu)
, u^{\otimes r} \bigr\rangle ,
\end{equation}
where $ F(\mu) := \mathbb{E} [ \mathscr U_\alpha f(\boldsymbol{\Sigma}^{1/2} Z + \mu) ] $ and where
$ Z $ is a standard $ d $-variate normal random vector. If $ Z' $ is
another such vector
independent of $ Z $, we can write
\[
F(\mu) = \mathbb{E} \bigl[ f \bigl(\bigl(\boldsymbol{\Sigma}^{1/2} Z +
\mu\bigr) \cos \alpha+ Z' \sin\alpha \bigr) \bigr] =
\int_{\mathbb{R}^d} f(\mu\cos\alpha+ \mathbf{Q}_\alpha z)
\phi_d(z) \,\mathrm{d}z ,
\]
where $ \mathbf{Q}_\alpha:= (\boldsymbol{\Sigma} \cos^2 \alpha+
\mathbf{I}_d \sin^2 \alpha)^{1/2} $.
Substituting $ y = \mathbf{Q}_\alpha^{-1} \mu\cos\alpha+ z $, we obtain
\[
F(\mu) =
\int_{\mathbb{R}^d} f(\mathbf{Q}_\alpha y) \phi_d
\bigl( y - \mathbf{Q}_\alpha^{-1} \mu\cos\alpha \bigr)
\,\mathrm{d}y .
\]
Differentiation yields
\begin{align*}
\bigl\langle\nabla^r F(\mu) , u^{\otimes r} \bigr\rangle &=
(-1)^r \cos^r \alpha
\int_{\mathbb{R}^d} f(\mathbf{Q}_\alpha y) \bigl\langle
\nabla^r \phi _d \bigl( y - \mathbf{Q}_\alpha^{-1}
\mu\cos\alpha \bigr) , v^{\otimes r} \bigr\rangle \,\mathrm{d}y
\\
&= (-1)^r \cos^r \alpha
\int_{\mathbb{R}^d} f(\mu\cos\alpha+ \mathbf{Q}_\alpha z) \bigl
\langle\nabla^r \phi_d(z) , v^{\otimes r} \bigr
\rangle \,\mathrm{d}z ,
\end{align*}
where $ v = \mathbf{Q}_\alpha^{-1} u $. By Lemma~\ref
{lm:IntDerphidBd}, we can estimate
\begin{equation}
\label{eq:OUDerNorm:cr} \bigl\llvert \bigl\langle\nabla^r F(\mu) ,
u^{\otimes r} \bigr\rangle \bigr\rrvert \le c_r \cos^r
\alpha M_0^*(f) \llvert v \rrvert ^r .
\end{equation}
Noting that $ \| \mathbf{Q}_\alpha^{-1} \| =
(\sigma^2 \cos^2 \alpha+ \sin^2 \alpha)^{-1/2} \le1/\sigma$ and plugging
into \eqref{eq:OUDerNorm:cr} and \eqref{eq:OUDerNorm:DerF} in turn,
the result follows.
\end{proof}

\begin{lemma}
\label{lm:OUDerNormApprox}
Let $ \mathscr A $ be a family of measurable sets in $ \mathbb{R}^d $,
which, along with
the underlying functions $ \rho_A $, meets Assumptions \textup{(A\ref
{A:first})--(A\ref{A:last})}.
Take an $ \mathbb{R}^d $-valued random vector $ W $, such that there
exist a vector
$ \mu\in\mathbb{R}^d $, a positive-definite matrix $ \boldsymbol{\Sigma} $ and a constant $ D \ge0 $,
such that for each $ A \in\mathscr A $,
\begin{equation}
\label{eq:OUDerNormApproxAss} \bigl\llvert \mathbb{P}(W \in A) - \mathcal{N}(\mu, \boldsymbol{
\Sigma}) (A) \bigr\rrvert \le D .
\end{equation}
Then for each $ \varepsilon> 0 $ and each $ f \in\{ f_A^\varepsilon,
f_A^{- \varepsilon} \} $,
where $ f_A^\varepsilon$ and $ f_A^{- \varepsilon} $ are as in
Lemma~\ref{lm:Smoothdelta}, we have
\begin{equation}
\label{eq:OUDerNormApproxBd}
\int_0^{\pi/2} \bigl\llvert \mathbb{E} \bigl(
\nabla^3 \mathscr U_\alpha f(W) \bigr) \bigr\rrvert
_\vee\tan\alpha \,\mathrm{d}\alpha \le \frac{c_3}{6 \sigma^3} + \sqrt{2(1 +
\kappa) c_1 c_3} \biggl( \frac{\gamma^*(\mathscr A
\mid\rho)}{\sigma} +
\frac{4 D}{\varepsilon} \biggr) .
\end{equation}
\end{lemma}

\begin{proof}
Fix $ A \in\mathscr A $ and $ \varepsilon> 0 $, and let $ f =
f_A^{\varepsilon} $ or $ f = f_A^{- \varepsilon} $.
In the first case, define $ A_1 := A $ and $ A_2 := A^{\varepsilon\mid
\rho} $, while in
the second case, define $ A_1 := A^{- \varepsilon\mid\rho} $ and $
A_2 := A $.

Similarly as observed in Remark~\ref{rk:ETensor},
$ \mathbb{E} ( \nabla^3 \mathscr U_\alpha f(W) ) $
is a tensor because it is an expectation of a random tensor. Since the latter
is symmetric, so is its expectation. By Proposition~\ref{pr:InjSym}, its
injective norm can be expressed as
\begin{equation}
\label{eq:OUDerNormApprox:InjSym} \bigl\llvert \mathbb{E} \bigl( \nabla^3 \mathscr
U_\alpha f(W) \bigr) \bigr\rrvert _\vee = \sup
_{|u| \le1} \bigl\llvert H_\alpha(u) \bigr\rrvert ,
\end{equation}
where
\begin{equation}
\label{eq:Halpha} H_\alpha(u) := \bigl\langle\mathbb{E} \bigl(
\nabla^3 \mathscr U_\alpha f(W) \bigr) ,
u^{\otimes3} \bigr\rangle = \mathbb{E} \bigl[ \bigl\langle
\nabla^3 \mathscr U_\alpha f(W) , u^{\otimes3} \bigr
\rangle \bigr] .
\end{equation}
Fix $ 0 < \beta< \pi/2 $ and $ u \in\mathbb{R}^d $ with $ |u| \le1 $.
We distinguish the cases $ 0 < \alpha\le\beta$ and $ \beta< \alpha
\le\pi/2 $.
In the first case, write, applying \eqref{eq:f2phi1},
\[
H_\alpha(u) = - \frac{\cos^3 \alpha}{\sin\alpha}
\int_{\mathbb{R}^d} F_\alpha (z) \bigl\langle\nabla
\phi_d(z) , u \bigr\rangle \,\mathrm{d}z ,
\]
where
\[
F_\alpha(z) := \mathbb{E} \bigl[ \bigl\langle\nabla^2 f(W
\cos\alpha+ z \sin\alpha) , u^{\otimes2} \bigr\rangle \bigr] .
\]
Notice that by Part~(\ref{SmoothdeltaM}) of Lemma~\ref
{lm:Smoothdelta} and Rademacher's theorem
(see Remark~\ref{rk:Rademacher}), $ \nabla^2 f $ is defined almost
everywhere. By Fubini's
theorem, the latter also holds for $ F_\alpha$.
% (the expectation of a random variable is defined
% if the latter is defined almost surely and is in $ L^1 $).
Moreover, where it is defined, we have, by Parts~(\ref
{Smoothdelta01+}) and (\ref{SmoothdeltaM}) of
Lemma~\ref{lm:Smoothdelta},
\begin{equation}
\label{eq:OUOdvNormAproksFz} \bigl\llvert F_\alpha(z) \bigr\rrvert \le
\frac{4(1 + \kappa)}{\varepsilon^2} \mathbb{P} ( W \cos\alpha+ z \sin\alpha\in A_2
\setminus A_1 ) .
\end{equation}
First, we estimate the right-hand side with $ W $ replaced by a $ d
$-variate normal random vector
with the same mean and covariance matrix. Lemma~\ref{lm:PerimNonStd} yields
\begin{equation}
\label{eq:OUDerNormApproxFzNormal} \mathcal{N} \bigl( \mu\cos\alpha+ z \sin\alpha, \boldsymbol{
\Sigma} \cos^2 \alpha \bigr) \{ A_2 \setminus
A_1 \} \le\frac{\gamma^*(\mathscr A \mid\rho) \varepsilon}{\sigma\cos
\alpha} .
\end{equation}
To estimate the remainder, combine \eqref{eq:OUDerNormApproxAss},
(A\ref{A:TrScale}), (A\ref{A:At})
and the fact that $ A_1 \subseteq A_2 $, resulting in
\begin{equation}
\label{eq:OUDerNormApproxFzNonNormal} \bigl\llvert \mathbb{P} ( W \cos\alpha+ z \sin\alpha\in
A_2 \setminus A_1 ) - \mathcal{N} \bigl( \mu\cos\alpha+ z
\sin\alpha, \boldsymbol{\Sigma} \cos^2 \alpha \bigr) \{
A_2 \setminus A_1 \} \bigr\rrvert \le2 D .
\end{equation}
Combining \eqref{eq:OUOdvNormAproksFz}, \eqref
{eq:OUDerNormApproxFzNormal} and
\eqref{eq:OUDerNormApproxFzNonNormal}, we obtain
\[
\bigl\llvert F_\alpha(z) \bigr\rrvert \le\frac{4(1 + \kappa)}{\varepsilon^2} \biggl(
\frac{\gamma^*(\mathscr A \mid\rho) \varepsilon}{\sigma\cos
\alpha} + 2 D \biggr) \le\frac{4(1 + \kappa)}{\varepsilon^2 \cos\alpha} \biggl( \frac{\gamma^*(\mathscr A \mid\rho) \varepsilon}{\sigma}
+ 2 D \biggr) .
\]
From Lemma~\ref{lm:IntDerphidBd}, it follows that
\begin{equation}
\label{eq:OUDerNormApproxSmall} \bigl\llvert H_\alpha(u) \bigr\rrvert \le
\frac{4(1 + \kappa) c_1 \cos^2 \alpha}{\varepsilon\sin\alpha
} \biggl( \frac{\gamma^*(\mathscr A \mid\rho)}{\sigma} + \frac
{2 D}{\varepsilon} \biggr) .
\end{equation}
Now we turn to the case $ \alpha\ge\beta$, where we estimate
$ |H_\alpha(u)| $ in a different way. First, we estimate the
right-hand side of
\eqref{eq:Halpha} with $ W $ replaced by a $ d $-variate normal random vector
with the same mean and covariance matrix. Lemma~\ref{lm:OUDerNorm} yields
\begin{equation}
\label{eq:OUDerNormApproxGzNormal} \bigl\llvert \bigl\langle\mathcal{N}(\mu, \boldsymbol{\Sigma})
\bigl\{ \nabla^3 \mathscr U_\alpha f \bigr\} ,
u^{\otimes3} \bigr\rangle \bigr\rrvert \le\frac{c_3 \cos^3 \alpha}{2 \sigma^3} .
\end{equation}
To estimate the remainder, write, applying \eqref{eq:f0phi3},
\begin{equation}
\label{eq:OUDerNormApproxGzNonNormal} H_\alpha(u) - \bigl\langle\mathcal{N}(\mu, \boldsymbol{
\Sigma}) \bigl\{ \nabla^3 \mathscr U_\alpha f \bigr\} ,
u^{\otimes3} \bigr\rangle = - \cot^3 \alpha
\int_{\mathbb{R}^d} G_\alpha(z) \bigl\langle\nabla^3
\phi_d(z) , u^{\otimes3} \bigr\rangle \,\mathrm{d}z ,
\end{equation}
where
\[
G_\alpha(z) := \mathbb{E} \bigl[ f(W \cos\alpha+ z \sin\alpha) \bigr] -
\mathcal{N} \bigl( \mu\cos\alpha+ z \sin\alpha, \boldsymbol{\Sigma}
\cos^2 \alpha \bigr) \{ f \} .
\]
Noting that $ 0 \le f \le1 $, write $ f(x) = \int_0^1 \mathbf{1}(x
\in\tilde A_t) \,\mathrm{d}t $, where
$ \tilde A_t := \{ x ;f(x) \ge t \} $. Consequently,
\begin{align*}
G_\alpha(z) &=
\int_0^1 \bigl[ \mathbb{P} ( W \cos\alpha+ z \sin
\alpha\in\tilde A_t ) - \mathcal{N} \bigl( \mu\cos\alpha+ z \sin
\alpha, \boldsymbol{\Sigma} \cos^2 \alpha \bigr) \{ \tilde
A_t \} \bigr] \,\mathrm{d}t
\\
&=
\int_0^1 \bigl[ \mathbb{P}( W \in\tilde
A_{t, \alpha, z}) - \mathcal{N}(\mu, \boldsymbol{\Sigma})\{ \tilde
A_{t, \alpha, z} \} \bigr] \,\mathrm{d}t ,
\end{align*}
where $ \tilde A_{t, \alpha, z} := (\tilde A_t - \sin\alpha z) \cos
^{-1} \alpha$.
By Part~(\ref{SmoothdeltaInt}) of Lemma~\ref{lm:Smoothdelta}, $
\tilde A_t \in\mathscr A
\cup\{ \varnothing , \mathbb{R}^d \} $ for all $ t \in(0, 1) $. By
Assumption~(A\ref{A:TrScale}), the
same is true for $ \tilde A_{t, \alpha, z} $. Therefore, $ |G_\alpha
(z)| \le D $
(observe that \eqref{eq:OUDerNormApproxAss} is trivially true for $ A
\in\{ \varnothing ,
\mathbb{R}^d \} $). Applying \eqref{eq:OUDerNormApproxGzNormal},
\eqref{eq:OUDerNormApproxGzNonNormal}
and Lemma~\ref{lm:IntDerphidBd}, we obtain
\begin{equation}
\label{eq:OUDerNormApproxLarge} \bigl\llvert H_\alpha(u) \bigr\rrvert \le
\frac{c_3 \cos^3 \alpha}{2 \sigma^3} + c_3 D \cot^3 \alpha.
\end{equation}
Taking the supremum over $ u $ in \eqref{eq:OUDerNormApproxSmall} and
\eqref{eq:OUDerNormApproxLarge}, applying \eqref
{eq:OUDerNormApprox:InjSym} and integrating,
we obtain
\begin{equation}
\label{eq:OUOdvNormApproxBdbeta} %
\begin{aligned}[b]
\int_0^{\pi/2} \bigl\llvert \mathbb{E} \bigl[
\nabla^3 \mathscr U_\alpha f(W) \bigr] \bigr\rrvert
_\vee\tan\alpha \,\mathrm{d}\alpha \le{}&
\int_0^\beta\frac{4(1 + \kappa) c_1 \cos\alpha}{\varepsilon
} \biggl(
\frac{\gamma^*(\mathscr A \mid\rho)}{\sigma} + \frac{2
D}{\varepsilon} \biggr) \,\mathrm{d}\alpha
\\
&{} + c_3
\int_\beta^{\pi/2} \biggl( \frac{\cos^2 \alpha\sin\alpha
}{2 \sigma^3} + D
\cot^2 \alpha\ \biggr) \,\mathrm{d}\alpha
\\
\le{}&\frac{4(1 + \kappa) c_1 \tan\beta}{\varepsilon} \biggl( \frac{\gamma^*(\mathscr A \mid\rho)}{\sigma} + \frac{2
D}{\varepsilon} \biggr)
\\
& {}+ \frac{c_3}{6 \sigma^3} + c_3 D \cot\beta .
\end{aligned} %
\end{equation}
Now choose $ \beta$ so that the sum of the terms with $ D $ is
optimal. This occurs at
$ \beta= \arctan ( \varepsilon\times\sqrt{\frac{c_3}{8(1 + \kappa)
c_1}} ) $.
Plugging into \eqref{eq:OUOdvNormApproxBdbeta}, we obtain \eqref
{eq:OUDerNormApproxBd},
completing the proof.
\end{proof}

Now we are ready to prove the main result.

\begin{proof}[Proof of Theorem~\ref{th:ClassBd}]
First, we prove the case where $ \mathscr A $ also meets (A\ref{A:TrScale}$'$).
Throughout the argument, fix $ \mathscr A $ along with the underlying
functions $ \rho_A $.
For each $ \beta_0 > 0 $, define
\begin{equation}
\label{eq:ClassBdAffKeta} K(\beta_0) := \sup\frac{ \llvert \mathbb{P}(W \in A) - \mathcal
{N}(0, \mathbf{I}_d) \{ A \} \rrvert }%
{\max \{ \sum_{i \in\mathscr I} \mathbb{E} \llvert X_i \rrvert ^3, \beta_0
\}} ,
\end{equation}
where the supremum runs over the family of all sums $ W = \sum_{i \in
\mathscr I} X_i $ of
independent random vectors with $ \mathbb{E}X_i = 0 $ and $
\operatorname{Var}(W) = \mathbf{I}_d $, and over all
$ A \in\mathscr A $. Now fix $ \beta_0 > 0 $, a sum $ W = \sum_{i
\in\mathscr I} X_i $
in the aforementioned family and a set $ A \in\mathscr A $. From
Lemma~\ref{lm:Smoothdelta},
it follows that
\begin{align*}
0& \le \mathcal{N}(0, \mathbf{I}_d) \bigl\{ f_A^{\varepsilon\mid\rho}
\bigr\} - \mathcal{N}(0, \mathbf{I}_d) \{ A \} \le \mathcal{N}(0,
\mathbf{I}_d) \bigl\{ A^{\varepsilon\mid\rho} \bigr\} - \mathcal{N}(0,
\mathbf{I}_d) \{ A \} \le \gamma^*(\mathscr A \mid\rho) \varepsilon ,
\\
0 &\le \mathcal{N}(0, \mathbf{I}_d) \{ A \} - \mathcal{N}(0,
\mathbf{I}_d) \bigl\{ f_A^{- \varepsilon\mid\rho} \bigr\} \le
\mathcal{N}(0, \mathbf{I}_d) \{ A \} - \mathcal{N}(0,
\mathbf{I}_d) \bigl\{ A^{- \varepsilon\mid\rho} \bigr\} \le \gamma^*(\mathscr A
\mid\rho) \varepsilon .
\end{align*}
Consequently,
\begin{align*}
\mathbb{P}(W \in A) - \mathcal{N}(0, \mathbf{I}_d) \{ A \} &\le
\mathbb{E}f_A^\varepsilon(W) - \mathcal{N}(0,
\mathbf{I}_d) \bigl\{ f_A^\varepsilon\bigr\} +
\gamma^*(\mathscr A \mid\rho) \varepsilon ,
\\
\mathbb{P}(W \in A) - \mathcal{N}(0, \mathbf{I}_d) \{ A \} &\ge
\mathbb{E}f_A^{- \varepsilon}(W) - \mathcal{N}(0, \mathbf
{I}_d) \bigl\{ f_A^{- \varepsilon} \bigr\} - \gamma^*(
\mathscr A \mid\rho) \varepsilon .
\end{align*}
Therefore,
\begin{equation}
\label{eq:ClassBdSmoothPerturb} %
\begin{aligned}[b] \bigl\llvert \mathbb{P}(W \in A) -
\mathcal{N}(0, \mathbf{I}_d) \{ A \} \bigr\rrvert \le{}&\max \bigl\{
\bigl\llvert \mathbb{E}f(W) - \mathcal{N}(0, \mathbf {I}_d) \{ f \}
\bigr\rrvert ;f \in\bigl\{ f_A^\varepsilon, f_A^{- \varepsilon}
\bigr\} \bigr\}
\\
& {}+ \gamma^*(\mathscr A \mid\rho) \varepsilon \end{aligned}
\end{equation}
Let $ f \in\{ f_A^\varepsilon, f_A^{- \varepsilon} \} $, and
let $ \tilde X_i $ and $ \theta$ be as in
Lemma~\ref{lm:Stein}. Applying \eqref{eq:SteinSlep} and Lemma~\ref
{lm:Stein} in turn, and
conditioning on $ X_i, \tilde X_i $ and $ \theta$, we obtain
\[
% \label{eq:ErrTiXEXX}
\mathbb{E}f(W) - \mathcal{N}(0, \mathbf{I}_d) \{ f \} = -
\int_0^{\pi/2} \sum_{i \in\mathscr I}
\mathbb{E} \bigl[ \bigl\langle T_i(\alpha) , X_i
\otimes\tilde X_i^{\otimes2 } - (1 - \theta) X_i^{\otimes3}
\bigr\rangle \bigr] \tan\alpha \,\mathrm{d}\alpha ,
\]
where
\[
% \label{eq:BETialpha}
T_i(\alpha) := \mathbb{E} \bigl[ \nabla^3
\mathscr U_\alpha f(W_i + \theta X_i) |
X_i, \tilde X_i, \theta \bigr]
\]
is a random tensor of order three. Now estimate
\begin{equation}
\label{eq:BdTialpha} \bigl\llvert \mathbb{E}f(W) - \mathcal{N}(0,
\mathbf{I}_d) \{ f \} \bigr\rrvert \le\sum
_{i \in\mathscr I} \mathbb{E} \biggl[
\int_0^{\pi/2} \bigl\llvert T_i(\alpha)
\bigr\rrvert _\vee \tan\alpha \,\mathrm{d}\alpha \bigl( \llvert
X_i \rrvert \llvert \tilde X_i \rrvert ^2 +
(1 - \theta) \llvert X_i \rrvert ^3 \bigr) \biggr] .
\end{equation}
To estimate $ \int_0^{\pi/2} \llvert T_i(\alpha) \rrvert _\vee\tan
\alpha \,\mathrm{d}\alpha$, we shall
use the conditional counterpart of Lemma~\ref{lm:OUDerNormApprox}
given $ X_i $, $ \tilde X_i $
and $ \theta$. To apply it, we need to estimate
\[
D_{i,A} := \bigl\llvert \mathbb{P}(W_i + \theta
X_i \in A \mid X_i, \tilde X_i, \theta) -
\mathcal{N}(\theta X_i, \boldsymbol{\Sigma}_i) \{ A \}
\bigr\rrvert ,
\]
where $ \boldsymbol{\Sigma}_i = \operatorname{Var}(W_i) $.
Assume that $ \boldsymbol{\Sigma}_i $ is non-singular. In this case, we
may write
\[
D_{i,A} = \bigl\llvert \mathbb{P} \bigl( \boldsymbol{
\Sigma}_i^{-1/2} W_i \in \boldsymbol{
\Sigma}_i^{-1/2}(A - \theta X_i) \mid
X_i, \tilde X_i, \theta \bigr) - \mathcal{N}(0,
\mathbf{I}_d) \bigl\{ \boldsymbol{\Sigma}_i^{-1/2}
(A - \theta X_i) \bigr\} \bigr\rrvert .
\]
To estimate $ D_{i,A} $, we apply the `bootstrapping' argument: we
refer to
\eqref{eq:ClassBdAffKeta} with
\newline
$ \boldsymbol{\Sigma}_i^{-1/2} W_i $ in place of
$ W $, noting independence of $ W_i $ and $ (X_i, \tilde X_i, \theta)
$, and observing that
$ \boldsymbol{\Sigma}_i^{-1/2} W_i $ is a sum of
independent random vectors with vanishing expectations
and with
\newline
$ \operatorname{Var} ( \boldsymbol{\Sigma}_i^{-1/2} W_i ) =
\mathbf{I}_d $.
Furthermore, observe that, given $ \theta$ and $ X_i $, we have
$ \boldsymbol{\Sigma}_i^{-1/2} (A - \theta X_i) \in\mathscr A $ by
(A\ref{A:TrScale}$'$).
Denoting by $ \sigma_i^2 $ the smallest eigenvalue of $ \boldsymbol{\Sigma}_i $
(with $ \sigma_i > 0 $), observe that $ \mathbb{E} \llvert \boldsymbol{\Sigma}_i^{-1/2} X_j \rrvert ^3
\le\sigma_i^{-3} \mathbb{E}|X_j|^3 $ (notice that $ \sigma_i \le1
$). By \eqref{eq:ClassBdAffKeta},
we have
\[
D_{i,A} \le K(\beta_0) \max \biggl\{ \frac{1}{\sigma_i^3}
\sum_{j \in\mathscr I \setminus\{ i \}} \mathbb{E} \llvert X_j
\rrvert ^3, \beta_0 \biggr\} \le \frac{K(\beta_0) \bar\beta}{\sigma_i^3} ,
\]
where $ \bar\beta:= \max \{ \sum_{j \in\mathscr I} \mathbb
{E}|X_j|^3, \beta_0 \} $.
Applying Lemma~\ref{lm:OUDerNormApprox} to the conditional
distribution of $ W $
given $ X_i $, $ \tilde X_i $ and $ \theta$, we find that
\[
\int_0^{\pi/2} \bigl\llvert T_i(\alpha)
\bigr\rrvert _\vee\tan\alpha \,\mathrm{d}\alpha \le B_i :=
\frac{c_3}{6 \sigma_i^3} + \sqrt{2(1 + \kappa) c_1 c_3} \biggl(
\frac{\gamma^*(\mathscr A \mid
\rho)}{\sigma_i} + \frac{4 K(\beta_0)}{\sigma_i^3} \frac{\bar\beta}{\varepsilon} \biggr) .
\]
Now \eqref{eq:BdTialpha} reduces to
\begin{equation}
\label{eq:BdBXXX} \bigl\llvert \mathbb{E}f(W) - \mathcal{N}(0,
\mathbf{I}_d) \{ f \} \bigr\rrvert \le\sum
_{i \in\mathscr I} B_i \mathbb{E} \bigl( \llvert
X_i \rrvert \llvert \tilde X_i \rrvert ^2 +
(1 - \theta) \llvert X_i \rrvert ^3 \bigr) \le
\frac{3}{2} \sum_{i \in\mathscr I} B_i
\mathbb{E} \llvert X_i \rrvert ^3 ,
\end{equation}
with the last inequality being due to H\"older's inequality.

Now fix $ 0 < \beta_* < 1 $ (an explicit value will be chosen later)
and assume first that
$ \bar\beta\le\beta_* $. By Jensen's inequality,
$ \mathbb{E}|X_i|^2 \le ( \mathbb{E}|X_i|^3 )^{2/3} \le
\bar\beta^{2/3} \le\beta_*^{2/3} $
for all $ i \in\mathscr I $. Next, for each unit vector $ u \in
\mathbb{R}^d $,
\[
\langle\boldsymbol{\Sigma}_i u , u \rangle = u^T
\boldsymbol{\Sigma}_i u = u^T \bigl(\mathbf{I}_d
- \mathbb{E}X_i X_i^T\bigr) u = 1 -
\mathbb{E}\langle X_i , u \rangle^2 \ge1 - \mathbb
{E} \llvert X_i \rrvert ^2 \ge1 - \beta_*^{2/3}
.
\]
Therefore, $ \sigma_i^2 \ge1 - \beta_*^{2/3} $ for all $ i \in
\mathscr I $.
In particular, the matrices $ \boldsymbol{\Sigma}_i $ are non-singular and
the quantities $ B_i $ can be uniformly bounded. Letting
$ \sigma_* := ( 1 - \beta_*^{2/3} )^{1/2} $, \eqref
{eq:BdBXXX} reduces to
\begin{equation}
\label{eq:ClassBdAffSmooth} \bigl\llvert \mathbb{E}f(W) - \mathcal{N}(0,
\mathbf{I}_d) \{ f \} \bigr\rrvert \le \biggl[ \frac{c_3}{4 \sigma_*^3} +
\sqrt{2(1 + \kappa) c_1 c_3} \biggl( \frac{3 \gamma^*(\mathscr A \mid\rho)}{2 \sigma_*} +
\frac{6 K(\beta_0)}{\sigma_*^3} \frac{\bar\beta}{\varepsilon} \biggr) \biggr] \bar\beta .
\end{equation}
Recalling \eqref{eq:ClassBdSmoothPerturb}, we obtain
\begin{equation*}
\begin{split} \bigl\llvert \mathbb{P}(W \in A) - \mathcal{N}(0,
\mathbf{I}_d) \{ A \} \bigr\rrvert \le{}& \biggl[ \frac{c_3}{4 \sigma_*^3} +
\sqrt{2(1 + \kappa) c_1 c_3} \biggl( \frac{3 \gamma^*(\mathscr A \mid\rho)}{2 \sigma_*} +
\frac{6 K(\beta_0)}{\sigma_*^3} \frac{\bar\beta}{\varepsilon} \biggr) \biggr] \bar\beta
\\
& {} + \gamma^*(\mathscr A \mid\rho) \varepsilon . \end{split}
\end{equation*}
Choosing $ \varepsilon:= 12 \bar\beta\sqrt{2(1 + \kappa) c_1
c_3}/\sigma_*^3 $, this reduces to
\begin{align}
\label{eq:ClassBdAffbeta<gama}
&\bigl\llvert \mathbb{P}(W \in A) - \mathcal{N}(0,
\mathbf{I}_d) \{ A \} \bigr\rrvert\nonumber\\
&\quad \le \biggl[ \frac{K(\beta_0)}{2} +
\frac{c_3}{4 \sigma_*^3} + \gamma^*(\mathscr A \mid\rho) \sqrt{2(1 + \kappa)
c_1 c_3} \biggl( \frac{3}{2 \sigma_*} + \frac{12}{\sigma_*^3}
\biggr) \biggr] \bar\beta .
\end{align}
Now we are left with the case $ \bar\beta\ge\beta_* $. We trivially estimate
\begin{equation}
\label{eq:ClassBdAffbeta>gama} \bigl\llvert \mathbb{P}(W \in A) - \mathcal{N}(0,
\mathbf{I}_d) \{ A \} \bigr\rrvert \le1 \le\frac{\bar\beta}{\beta_*} .
\end{equation}
Dividing estimates \eqref{eq:ClassBdAffbeta<gama} and \eqref
{eq:ClassBdAffbeta>gama}
by $ \bar\beta$, taking the supremum over all $ A \in\mathscr A $
and all sums $ W $,
and plugging into \eqref{eq:ClassBdAffKeta}, we obtain
\[
K(\beta_0) \le\max \biggl\{ \frac{1}{\beta_*}, \frac{K(\beta_0)}{2} +
\frac{c_3}{4 \sigma_*^3} + \gamma^*(\mathscr A \mid\rho) \sqrt{2(1 + \kappa)
c_1 c_3} \biggl( \frac{3}{2 \sigma_*} + \frac{12}{\sigma_*^3}
\biggr) \biggr\} .
\]
Since $ K(\beta_0) \le1/\beta_0 < \infty$, it follows that
\begin{equation}
\label{eq:ClassBdAffKetaBd} K(\beta_0) \le\max \biggl\{ \frac{1}{\beta_*},
\frac{c_3}{2 \sigma_*^3} + \gamma^*(\mathscr A \mid\rho) \sqrt{2(1 + \kappa)
c_1 c_3} \biggl( \frac{3}{\sigma_*} + \frac{24}{\sigma_*^3}
\biggr) \biggr\} .
\end{equation}
Choose $ \beta_* := 1/27 $, which is approximately optimal for the
class of
all half-lines on the real line. Straightforward numerical estimation yields
$ K(\beta_0) \le\max \{ 27, 1 + 50 \gamma^*(\mathscr A \mid
\rho)
\sqrt{1 + \kappa} \} $;
%.%
%
% R:
%
% b <- 1/27
% s <- sqrt(1 - b^(2/3))
% c1 <- 2/sqrt(2*pi)
% c3 <- (2 + 8/exp(3/2))/sqrt(2*pi)
% c3/2/s^3
% ## [1] 0.900906
% sqrt(2*c1*c3)*(3/s + 24/s^3)
% ## [1] 49.39385
%
this holds true for all $ \beta_0 > 0 $. Thus,
for a fixed sum $ W = \sum_{i \in\mathscr I} X_i $, one can plug the preceding
estimate into \eqref{eq:ClassBdAffKeta}, choosing $ \beta_0 := \sum_{i \in\mathscr I} \mathbb{E}|X_i|^3 $;
\eqref{eq:ClassBdAff} follows.

Now we turn to the case where $ \mathscr A $ does not necessarily meet
Assumption~(A\ref{A:TrScale}$'$).
This time, fix $ \kappa\ge0 $ and for each $ \beta_0, \gamma_0 > 0
$, define
\begin{equation}
\label{eq:ClassBdKeta} K(\beta_0, \gamma_0) := \sup
\frac{ \llvert \mathbb{P}(W \in A) -
\mathcal{N}(0, \mathbf{I}_d) \{ A \} \rrvert }%
{
\max \{ \sum_{i \in\mathscr I} \mathbb{E} \llvert X_i \rrvert ^3, \beta_0
\}
\max \{ \gamma^*(\mathscr A \mid\rho), \gamma_0 \}
} ,
\end{equation}
where the supremum runs over the family of all sums $ W = \sum_{i \in
\mathscr I} X_i $ of independent
random vectors with
$ \mathbb{E}X_i = 0 $ and $ \operatorname{Var}(W) = \mathbf{I}_d $,
all classes $ \mathscr A $ which, along with the underlying
functions $ \rho_A $, satisfy Assumptions~(A\ref{A:first})--(A\ref
{A:last}) (with the chosen $ \kappa$),
and all $ A \in\mathscr A $.

Now fix $ \beta_0, \gamma_0 > 0 $, a sum $ W = \sum_{i \in\mathscr
I} X_i $ in the aforementioned family,
a class $ \mathscr A $ along with functions $ \rho_A $ satisfying
Assumptions~(A\ref{A:first})--(A\ref{A:last}),
and a set $ A \in\mathscr A $.
We proceed as in the previous case up to the estimation of $ D_{i,A} $.
For the latter, we now refer to \eqref{eq:ClassBdKeta}, again with
$ \boldsymbol{\Sigma}_i^{-1/2} W_i $ in place of $ W $. However, the set
$ \boldsymbol{\Sigma}_i^{-1/2} (A - \theta X_i) $ might not be in $
\mathscr A $. Instead, it is in the
class $ \tilde{\mathscr A} := \{ \boldsymbol{\Sigma}_i^{-1/2} A'
;A' \in\mathscr A \} $.
Thus, we may take $ \boldsymbol{\Sigma}_i^{-1/2} (A - \theta X_i) $ in
place of $ A $
provided that we take $ \tilde{\mathscr A} $ in place of $ \mathscr A $.
By Lemma~\ref{lm:ClassTrans}, we may take the latter provided that we take
the underlying family of functions $ \tilde\rho_{\tilde A^\prime}(x) :=
\rho_{\boldsymbol{\Sigma}_i^{1/2} \tilde A^\prime}(\boldsymbol{\Sigma}_i^{1/2} x) $,
$ \tilde A^\prime\in\tilde{\mathscr A} $, in place of the family $
\rho_{A^\prime} $,
$ A^\prime\in\mathscr A $: in this case, $ \kappa$ stays the same.
Denoting by $ \sigma_i^2 $ the smallest eigenvalue of $ \boldsymbol{\Sigma}_i $
(with $ \sigma_i > 0 $), recall that $ \mathbb{E} \llvert \boldsymbol{\Sigma}_i^{-1/2} X_j \rrvert ^3
\le\sigma_i^{-3} \mathbb{E}|X_j|^3 $ and observe that, again by
Lemma~\ref{lm:ClassTrans},
$ \gamma^*(\tilde{\mathscr A} \mid\tilde\rho) \le\gamma
^*(\mathscr A \mid\rho)/\sigma_i $
(notice that $ \sigma_i \le1 $). By \eqref{eq:ClassBdKeta}, we have
\[
D_{i,A} \le K(\beta_0, \gamma_0) \max \biggl\{
\frac{1}{\sigma_i^3} \sum_{j \in\mathscr I \setminus\{ i \}} \mathbb{E} \llvert
X_j \rrvert ^3, \beta_0 \biggr\} \max
\biggl\{ \frac{\gamma^*(\mathscr A \mid\rho)}{\sigma_i}, \gamma_0 \biggr\} \le
\frac{K(\beta_0, \gamma_0) \bar\beta\bar\gamma}{\sigma_i^4} ,
\]
where $ \bar\beta:= \max \{ \sum_{j \in\mathscr I} \mathbb
{E}|X_j|^3, \beta_0 \} $
and $ \bar\gamma:= \max \{ \gamma^*(\mathscr A \mid\rho),
\gamma_0 \} $.
Applying Lemma~\ref{lm:OUDerNormApprox} to the conditional
distribution of $ W $
given $ X_i $, $ \tilde X_i $ and $ \theta$, we find that
\[
\int_0^{\pi/2} \bigl\llvert T_i(\alpha)
\bigr\rrvert _\vee\tan\alpha \,\mathrm{d}\alpha \le B_i :=
\frac{c_3}{6 \sigma_i^3} + \sqrt{2(1 + \kappa) c_1 c_3} \biggl(
\frac{\bar\gamma}{\sigma
_i^2} + \frac{4 K(\beta_0, \gamma_0)}{\sigma_i^4} \frac{\bar\beta\bar
\gamma}{\varepsilon} \biggr) .
\]
Again, fix $ 0 < \beta_* < 1 $, let $ \sigma_* := ( 1 - \beta
_*^{2/3} )^{1/2} $ and assume first that
$ \bar\beta\le\beta_* $. By the same argument as in the first part,
we derive
\[
\bigl\llvert \mathbb{P}(W \in A) - \mathcal{N}(0, \mathbf{I}_d) \{ A
\} \bigr\rrvert \le \biggl[ \frac{c_3}{4 \sigma_*^3} + \sqrt{2(1 + \kappa)
c_1 c_3} \biggl( \frac{3 \bar\gamma}{2 \sigma_*^2} + \frac{6 K(\beta_0, \gamma_0)}{\sigma_*^4}
\frac{\bar\beta\bar
\gamma}{\varepsilon} \biggr) \biggr] \bar\beta + \bar\gamma\varepsilon .
\]
Choosing $ \varepsilon:= 12 \bar\beta\sqrt{2 (1 + \kappa)c_1
c_3}/\sigma_*^4 $, this reduces to
\begin{equation}
\label{eq:ClassBdbeta<gama} %
\begin{aligned}[b]
&\bigl\llvert \mathbb{P}(W \in A) -
\mathcal{N}(0, \mathbf{I}_d) \{ A \} \bigr\rrvert \\
&\quad \le
\frac{c_3 \bar\beta}{4 \sigma_*^3} + \biggl[ \frac{K(\beta_0, \gamma_0)}{2} + \sqrt{2(1 + \kappa)
c_1 c_3} \biggl( \frac{3}{2 \sigma_*^2} + \frac{12}{\sigma_*^4}
\biggr) \biggr] \bar\beta\bar\gamma
\\
&\quad \le \biggl[ \frac{K(\beta_0, \gamma_0)}{2} + \frac{c_3}{4 \gamma_0 \sigma_*^3} + \sqrt{2(1 + \kappa)
c_1 c_3} \biggl( \frac{3}{2 \sigma_*^2} + \frac{12}{\sigma_*^4}
\biggr) \biggr] \bar\beta\bar\gamma . \end{aligned} %
\end{equation}
In the case $ \bar\beta\ge\beta_* $, we trivially estimate
\begin{equation}
\label{eq:ClassBdbeta>gama} \bigl\llvert \mathbb{P}(W \in A) - \mathcal{N}(0,
\mathbf{I}_d) \{ A \} \bigr\rrvert \le1 \le\frac{\bar\beta\bar\gamma}{\beta_* \gamma_0} .
\end{equation}
Divide the estimates \eqref{eq:ClassBdAffbeta<gama} and \eqref
{eq:ClassBdAffbeta>gama}
by $ \bar\beta\bar\gamma$ and take the supremum over all $ A \in
\mathscr A $, all sums $ W $,
and all families $ \mathscr A $ (along with functions $ \rho_A $).
Plugging into
\eqref{eq:ClassBdKeta}, we obtain
\[
K(\beta_0, \gamma_0) \le\max \biggl\{
\frac{1}{\beta_* \gamma_0}, \frac{K(\beta_0, \gamma_0)}{2} + \frac{c_3}{4 \sigma_*^3 \gamma_0} + \sqrt{2(1 + \kappa)
c_1 c_3} \biggl( \frac{3}{2 \sigma_*^2} + \frac{12}{\sigma_*^4}
\biggr) \biggr\} .
\]
Since $ K(\beta_0, \gamma_0) \le1/(\beta_0 \gamma_0) < \infty$,
it follows that
\begin{equation}
\label{eq:ClassBdKetaBd} K(\beta_0, \gamma_0) \le\max \biggl\{
\frac{1}{\beta_* \gamma_0}, \frac{c_3}{2 \sigma_*^3 \gamma_0} + \sqrt{2(1 + \kappa) c_1
c_3} \biggl( \frac{3}{\sigma_*^2} + \frac{24}{\sigma_*^4} \biggr) \biggr\}
.
\end{equation}
As in the first case, choose $ \beta_* := 1/27 $. Straightforward numerical
estimation yields $ K(\beta_0, \gamma_0) \le\max \{ 27/\gamma
_0, 1/\gamma_0 + 53
\sqrt{1 + \kappa} \} $;
%.%
%
% R:
%
% b <- 1/27
% s <- sqrt(1 - b^(2/3))
% c1 <- 2/sqrt(2*pi)
% c3 <- (2 + 8/exp(3/2))/sqrt(2*pi)
% c3/2/s^3
% ## [1] 0.900906
% sqrt(2*c1*c3)*(3/s^2 + 24/s^4)
% ## [1] 52.39009
%
this holds true for all $ \beta_0, \gamma_0 > 0 $.
Thus, for a fixed sum $ W = \sum_{i \in\mathscr I} X_i $ and a fixed
class $ \mathscr A $
along with functions $ \rho_A $, one can plug the preceding estimate
into \eqref{eq:ClassBdKeta},
choosing $ \beta_0 := \sum_{i \in\mathscr I} \mathbb{E}|X_i|^3 $ and
$ \gamma_0 := \gamma^*(\mathscr A \mid\rho) $; \eqref{eq:ClassBd} follows.
This completes the proof.
\end{proof}

\section{Derivation of the bound on the Gaussian perimeter of convex sets}
\label{sc:Perim}

In this section, we prove Theorem~\ref{th:PerimBd}, and also state and prove
Proposition~\ref{pr:PerimAnn}, which is a generalization of
Proposition~\ref{pr:PerimAnnNorm}.
Throughout this section, fix $ d \in\mathbb{N}$ and denote by $
\mathscr C_d $ the class of
all measurable convex sets in $ \mathbb{R}^d $. From Section~\ref
{sc:Intr}, recall the
definitions of $ \delta_A $ and $ A^t $ for a set $ A \subseteq
\mathbb{R}^d $.
Recall also that $ \mathscr H^r $ denotes the $ r $-dimensional
Hausdorff measure.

The first result of the section is closely related to Lemma~11 of
Livshyts~\cite{Liv14}.

\begin{proposition}
\label{pr:PerimAnn}
Let $ \mathscr A $ be a class of certain convex sets in $ \mathbb{R}^d $.
Suppose that $ A^t \in\mathscr A \cup\{ \varnothing \} $ for all $ A
\in\mathscr A $ and all
$ t \in\mathbb{R}$. Take a continuous function
$ f \colon\mathbb{R}^d \to[0, \infty) $, which is integrable with
respect to the
Lebesgue measure. Then we have $ \gamma_f(\mathscr A) = \gamma
_f^*(\mathscr A) $, where
\begin{align*}
\gamma_f(\mathscr A) &= \sup \biggl\{
\int_{\partial A} f(x) \mathscr H^{d-1}(\mathrm{d}x) ;A \in
\mathscr A \biggr\} ,
\\
\gamma_f^*(\mathscr A) &= \sup \biggl\{ \frac{1}{\varepsilon}
\int_{A^\varepsilon\setminus A} f(x)\, \mathrm {d}x, \frac{1}{\varepsilon}
\int_{A \setminus A^{- \varepsilon}} f(x) \,\mathrm{d}x ; \varepsilon> 0, A \in
\mathscr A \biggr\} .
\end{align*}
\end{proposition}

% \begin{remark}
% The preceding assertion is closely related to Lemma~11 of Livshyts~
%\cite{Liv14}:
% the former would follow from the latter if the part with $ A^\eps
%\setminus A $ were not
% included in the definition of $ \gamma_f^*(\mathscr A) $.
% \end{remark}

Before proving the preceding assertion, we need to introduce some
notation and auxiliary
results. For a map $ g \colon A \to\mathbb{R}^n $, where $ A
\subseteq\mathbb{R}^d $ is a measurable
set, and for a point $ x \in A $ where $ g $ is differentiable, denote
by $ \mathbf{D}g(x) $ its
derivative (i.e., Jacobian matrix) at $ x $. For each $ r = 0, 1, 2,
\ldots$\,, define
$ J_r g(x) $, the $ r $-dimensional absolute Jacobian, as follows: if $
\operatorname{rank}\mathbf{D}g(x) < r $,
set $ J_r g(x) := 0 $. If $ \operatorname{rank}\mathbf{D}g(x) > r $,
set $ J_r g(x) := \infty$.
Finally, if $ \operatorname{rank}\mathbf{D}g(x) = r $, define $ J_r
g(x) $ to be the product of $ r $
non-zero singular values in the singular-value decomposition of $
\mathbf{D}g(x) $,
that is, $ \mathbf{D}g(x) = \mathbf{U} \boldsymbol{\Sigma}\mathbf{V}
$, where
$ \mathbf{U} $ and $ \mathbf{V} $ are orthogonal matrices and where $
\boldsymbol{\Sigma}$
is a diagonal rectangular matrix with non-negative diagonal elements
referred to as singular
values. It is easy to see that the definition is independent of the
decomposition.
Notice that for $ n = 1 $, we have $ J_1 g(x) = |\nabla g(x)| $.

The main tool used in the proof of Proposition~\ref{pr:PerimAnn} will
be the
following assertion, which can be regarded as a curvilinear variant of
Fubini's theorem.
As a special case, it also includes the change of variables formula in
the multi-dimensional
integral.

\begin{proposition}[Federer~\cite{Fed}, Corollary~3.2.32]
\label{pr:curviFubini}
Let $ A \subseteq\mathbb{R}^d $ be a measurable set, $ f \colon A \to
\mathbb{R}$
a measurable function and $ g \colon A \to\mathbb{R}^n $ a locally
Lipschitzian map.
Take $ 0 \le r \le d $ and assume that $ f J_r g $ is integrable with respect
to the Lebesgue measure. Then $ f |_{g^{-1}(\{ y \})} $
is $ \mathscr H^{d-r} $-integrable
for almost all $ y $ with respect to $ \mathscr H^r $, the function
$ y \mapsto\int_{g^{-1}( \{ y \} )} f \,\mathrm{d}\mathscr H^{d-r} $
is measurable and
\[
\int_{ \{ x \in A ;\mathscr H^{d-r}(g^{-1}( \{ g(x) \} )) > 0 \} } f(x) J_r g(x) \,\mathrm{d}x =
\int_{\mathbb{R}^n}
\int_{g^{-1}( \{ y \} )} f \,\mathrm{d}\mathscr H^{d-r} \mathscr
H^r(\mathrm{d}y) .
\]
\end{proposition}

\begin{remark}
The integrand in the left-hand side is defined for almost all $ x \in A
$, because
$ g $ is almost everywhere differentiable by Rademacher's theorem.
\end{remark}

\begin{corollary}[Coarea Formula]
\label{co:coarea}
Let $ d $, $ n $, $ A $, $ f $ and $ g $ be as in the preceding
statement. Suppose that $ d \ge n $.
Then we have
\[
\int_A f(x) J_n g(x) \,\mathrm{d}x =
\int_{\mathbb{R}^n}
\int_{g^{-1}( \{ y \} )} f \,\mathrm{d}\mathscr H^{d-n} \,\mathrm{d}y .
\]
\end{corollary}

\begin{proof}
Apply Proposition~\ref{pr:curviFubini} with $ d = n $ and observe that
by the implicit function theorem, $ J_n g(x) > 0 $ implies $ \mathscr
H^{d-n}(g^{-1}( \{ g(x) \} )) > 0 $.
\end{proof}

Now we turn to some simple properties of convex sets. First, one can
easily check that
if $ C $ is a non-empty convex set and $ x \in\mathbb{R}^d $, there
exists a unique point
in $ \overline{C} $ which is closest to $ x $.

\begin{definition}
The \emph{orthogonal projection} to a non-empty convex set $ C $ is a map
$ p^\perp_C \colon\mathbb{R}^d \to\overline{C} $, where $ p^\perp
_C(x) $ is defined to be
the unique point in $ \overline{C} $ which is closest to $ x $.
\end{definition}
\begin{proposition}
\label{pr:ConvEl}
Let $ C $ be a convex set, which is neither the empty set nor the whole
$ \mathbb{R}^d $.
\begin{enumerate}[$(1)$]
\item\label{ConvEl:Incrdistf} For each $ x \in\mathbb{R}^d $ and
each $ \varepsilon> 0 $, there
exists $ y \in\mathbb{R}^d $ with $ 0 < \delta_C(y) - \delta_C(x) =\break
|y - x| < \varepsilon$.
\item\label{ConvEl:Diff} $ \delta_C $ is almost everywhere differentiable.
\item\label{ConvEl:Diff1} For each $ x $ where $ \delta_C $ is
differentiable, we have
$ |\nabla\delta_C(x)| = 1 $.
% \item\label{ConvEl:bd0} The boundary of $ C $ has Lebesgue measure
%zero.
%
\item\label{ConvEl:bdCt} For each $ t \in\mathbb{R}$, we have $
\partial C^t = \{ x ;\delta_C(x) = t \} $.
\end{enumerate}
\end{proposition}

\begin{proof}
If $ x \in\operatorname{Int}C $, there exists a point $ z \in\mathbb
{R}^d \setminus\operatorname{Int}C $ which is
closest to $ x $. For all $ y = (1 - \tau) x + \tau z $,
where $ 0 \le\tau\le1 $, we have $ \operatorname{dist}(y, \mathbb{R}^d \setminus C) =
\operatorname{dist}(x, \mathbb{R}^d \setminus C) - |x - y| $, that is,
$ - \delta_C(y) = - \delta_C(x) - |x - y| $. Next, if $ x \in\mathbb
{R}^d \setminus\overline{C} $,
take $ \tau \ge 0 $ and let $ y = (1 + \tau) x - \tau p^\perp_C(x) $. By convexity, we have
$ \langle w - p^\perp_C(x) , x - p^\perp_C(x) \rangle \le0 $
for all $ w \in C $. As a result,
$ \operatorname{dist}(y, C) = \operatorname{dist}(x, C) + |x - y| $
for all $ \tau\ge0 $. Finally, if $ x \in\partial C $,
it is well known that there exist a unit outer normal vector $ u $
(possibly more than one); then,
for all $ y = x + \tau u $, where $ \tau\ge0 $, we again have
$ \operatorname{dist}(y, C) = \operatorname{dist}(x, C) + |x - y| $.
This proves (\ref{ConvEl:Incrdistf}).

One can easily check that $ \delta_C $ is non-expansive. By
Rademacher's theorem
(see also Remark~\ref{rk:Rademacher}), it is almost everywhere
differentiable and
$ |\nabla\delta_C(x)| \le1 $ for all $ x $ where it is
differentiable. This proves
(\ref{ConvEl:Diff}). However, by (\ref{ConvEl:Incrdistf}), we have
$ |\nabla\delta_C(x)| \ge1 $. This proves (\ref{ConvEl:Diff1}).

% Now consider the function $ x \mapsto\dist(x, C) = \max\{
%\distf_C(x), 0 \} $.
% Like $ \distf_C $, this function is non-expansive and therefore
%almost everywhere differentiable.
% However, on the boundary, it is not: if $ x \in\bd C $ and $ n $ is
%a normal vector, then
% $ \dist(x + \tau n, C) = \tau$ for all $ \tau\ge0 $ and $ \dist(x
%+ \tau n, C) \ge0 $
% for all $ \tau\in\RR$. Applying Rademacher's theorem, we prove
%\Tref{ConvEl:bd0}.

From the continuity of $ \delta_C $, it follows that $ \partial C^t
\subseteq\{ x ;\delta_C(x) = t \} $.
The opposite follows from (\ref{ConvEl:Incrdistf}). This proves (\ref
{ConvEl:bdCt}).
\end{proof}

\begin{proof}[Proof of Proposition~\ref{pr:PerimAnn}]
Without loss of generality, we may assume that $ \varnothing $ and $
\mathbb{R}^d $ are not elements of
$ \mathscr A $. Take $ A \in\mathscr A $. By the Coarea formula, we have
\[
\int_{A^\varepsilon\setminus A} f(x) J_1 \delta_A(x)
\,\mathrm{d}x =
\int_0^\varepsilon
\int_{\delta_A^{-1}(\{ t \})} f(x) \mathscr H^{d-1}(\mathrm{d}x)
\,\mathrm{d}t .
\]
Applying Parts~(\ref{ConvEl:Diff1}) and (\ref{ConvEl:bdCt}) of
Proposition~\ref{pr:ConvEl}, this
reduces to
\[
\int_{A^\varepsilon\setminus A} f(x) \,\mathrm{d}x =
\int_0^\varepsilon
\int_{\partial A^t} f(x) \mathscr H^{d-1}(\mathrm {d}x)
\,\mathrm{d}t \le \varepsilon \gamma_f(\mathscr A) .
\]
Similarly, we obtain
\[
\int_{A \setminus A^{- \varepsilon}} f(x) \,\mathrm{d}x =
\int_{- \varepsilon}^0
\int_{\partial A^t} f(x) \mathscr H^{d-1}(\mathrm{d}x)
\,\mathrm{d}t \le \varepsilon \gamma_f(\mathscr A)
\]
(remember that $ A^t \in\mathscr A \cup\{ \varnothing \} $; for $ A
= \varnothing $, the inner
integral vanishes).
Dividing by $ \varepsilon$, and taking the supremum over $ \varepsilon
$ and $ A $, we obtain $ \gamma_f^*(\mathscr A)
\le\gamma_f(\mathscr A) $.

To prove the opposite inequality, observe first that, by Parts~(\ref{ConvEl:Diff}) and (\ref{ConvEl:Diff1}) of Proposition~\ref{pr:ConvEl}, $ p^\perp_A $ is
non-expansive.
Next, observe that $ p^\perp_A ( (1 + \tau) x - \tau p^\perp
_A(x) ) = p^\perp_A(x) $
for all $ x \in\mathbb{R}^d \setminus\overline{A} $ and all $ \tau
\ge0 $. Therefore,
if $ p^\perp_A $ is differentiable at $ x \in\mathbb{R}^d \setminus
\overline{A} $, we have
$ \operatorname{rank}\mathbf{D}p^\perp_A(x) \le d - 1 $ and,
moreover, $ J_{d-1} p^\perp_A(x) \le1 $.
By Proposition~\ref{pr:curviFubini}, we have
\begin{align*}
\int_{A^\varepsilon\setminus A} f(x) \,\mathrm{d}x &\ge
\int_{\{ x \in A^\varepsilon\setminus\overline A ;\mathscr
H^1((p^\perp_A)^{-1}(\{ p^\perp_A(x) \})) > 0 \}} f(x) J_{d-1} p^\perp_A(x)
\,\mathrm{d}x
\\
&=
\int_{\partial A}
\int_{(p^\perp_A)^{-1}(\{ y \} \cap(A^\varepsilon
\setminus\overline A))} f \,\mathrm{d}\mathscr H^1 \mathscr
H^{d-1}(\mathrm{d}y) .
\end{align*}
If $ u $ is a unit outer normal vector at $ y \in\partial C $, then $
p^\perp_A(y + \tau u) = y $
for all $ \tau\ge0 $. Moreover, $ y + \tau u \in(p^\perp_A)^{-1}(\{
y \}) \cap(A^\varepsilon\setminus\overline A) $
for all $ 0 < \tau\le\varepsilon$. Therefore,
$ \mathscr H^1 ( (p^\perp_A)^{-1}(\{ y \}) \cap(A^\varepsilon
\setminus\overline A) )
\ge\varepsilon$. As a result,
\[
\int_{A^\varepsilon\setminus A} f(x) \,\mathrm{d}x \ge \varepsilon
\int_{\partial A} f^{-\varepsilon}(y) \mathscr H^{d-1}(
\,\mathrm{d}y) ,
\]
where $ f^{- \varepsilon}(x) := \inf_{|v| \le\varepsilon} f(x+v) $.
Dividing by $ \varepsilon$, we obtain
\[
\int_{\partial A} f^{-\varepsilon}(y) \mathscr H^{d-1}(
\,\mathrm{d}y) \le\gamma_f^*(\mathscr A) .
\]
Since $ f $ is continuous, we have $ \lim_{\varepsilon\downarrow0}
f^{- \varepsilon}(x) = f(x) $
for all $ x \in\mathbb{R}^d $. Applying the dominated convergence
theorem and taking
the supremum over all $ A $, we obtain $ \gamma_f(\mathscr A) \le
\gamma_f^*(\mathscr A) $.
This completes the proof.
\end{proof}

The orthogonal projection will be one of two key maps used in the proof
of Theorem~\ref{th:PerimBd}.
The other one will be the \emph{radial projection}.

\begin{definition}
Let $ C $ be a convex set with $ 0 \in\operatorname{Int}C $. We
define the \emph{radial function} of $ C $
to be the map $ \rho_C \colon\mathbb{R}^d \setminus\{ 0 \} \to(0,
\infty] $ defined by
\[
\rho_C(x) := \sup \biggl\{ r > 0 ;r \frac{x}{ \llvert x \rrvert } \in C \biggr
\} = \inf \biggl\{ r > 0 ;r \frac{x}{ \llvert x \rrvert } \notin C \biggr\}
\]
and the \emph{radial projection} of $ C $ to be the map
$ p^\rho_C \colon\{ x \in\mathbb{R}^d \setminus\{ 0 \} ;\rho_C(x)
< \infty\} \to\partial C $ defined by
$ \displaystyle p^\rho_C(x) := \rho_C(x) \frac{x}{|x|} $.
\end{definition}
\begin{lemma}
\label{lm:rad}
Let $ C $ be as before. Define the set $ D := \{ x \in\mathbb{R}^d
\setminus\{ 0 \} ;\rho_C(x) < \infty\} $.
Then:
\begin{enumerate}[$(1)$]
\item\label{rad:Lip} $ D $ is open and $ \rho_C $ and $ p^\rho_C $
are locally Lipschitzian on $ D $.
\item\label{rad:J} If $ \rho_C $ is differentiable at $ x $, so is $
p^\rho_C $,
there is a unique outer unit normal vector at $ p^\rho_C(x) $ and we have
\[
J_{d-1} p^\rho_C(x) = \biggl( \frac{\rho_C(x)}{|x|}
\biggr)^{d-1} \frac{1}{\cos\theta} ,
\]
where $ \theta$ is the angle between $ x $ and the outer unit normal
vector at
$ p^\rho_C(x) $.
\end{enumerate}
\end{lemma}

\begin{proof}
Since $ 0 \in\operatorname{Int}C $, there exists $ r_0 > 0 $, such
that $ \{ y \in\mathbb{R}^d ;|y| < r_0 \}
\subseteq C $. Fix $ x \in\mathbb{R}^d \setminus\{ 0 \} $. Let $ r_1
:= \rho_C(x) $ and $ v := x/|x| $.
Take $ w \perp v $ and $ s, t \in\mathbb{R}$, and let $ z := s v + t
w $. By convexity, $ z \in C $ if
$ 0 \le s < r_1 ( 1 - \frac{|t|}{r_0} ) $, and
$ z \notin C $ if $ s > r_1 ( 1 + \frac{|t|}{r_0} ) $.
Consequently,
\[
\frac{|z|}{\frac{s}{r_1} + \frac{|t|}{r_0}} \le\rho_C(z) \le\frac
{|z|}{\frac{s}{r_1} - \frac{|t|}{r_0}} ,
\]
provided that $ s > 0 $ and $ |t| < s r_0/r_1 $. Letting $ s = |x| +
\sigma$ and $ t = \tau$, we obtain
\[
\frac{\sqrt{ ( 1 + \frac{\sigma}{|x|} )^2 + (
\frac{\tau}{|x|} )^2}}%
{1 + \frac{\sigma}{|x|} + \frac{\rho_C(x)}{r_0} \frac{|\tau|}{|x|}} \le \frac{\rho_C(x + \sigma v + \tau w)}{\rho_C(x)} \le \frac{\sqrt{ ( 1 + \frac{\sigma}{|x|} )^2 + (
\frac{\tau}{|x|} )^2}}%
{1 + \frac{\sigma}{|x|} - \frac{\rho_C(x)}{r_0} \frac{|\tau|}{|x|}} ,
\]
provided that $ \sigma> - |x| $ and $ |\tau| < (|x| + \sigma)
r_0/\rho_C(x) $. From the
preceding inequality, we deduce first that $ D $ is open, then that $
\rho_C $ is continuous on
$ D $, then that $ \rho_C $ is locally Lipschitzian on $ D $ and
finally that the latter also holds
for $ p^\rho_C $. This proves (\ref{rad:Lip}).

Now suppose that $ \rho_C $ is differentiable at $ x $. By the chain
rule, so is
$ p^\rho_C $ and straightforward computation yields
\begin{equation}
\label{eq:Diffradp} \mathbf{D}p^\rho_C(x) v = \bigl\langle
\nabla\rho_C(x) , v \bigr\rangle \frac{x}{|x|} +
\rho_C(x) \biggl( \frac{v}{|x|} - \frac{\langle x , v \rangle
x}{|x|^3} \biggr) .
\end{equation}
Observe that since $ p^\rho_C(kx) = p^\rho_C(x) $ for all $ k > 0 $,
we have,
by the chain rule, $ \mathbf{D}p^\rho_C(kx) = \frac{1}{k} \mathbf
{D}p^\rho_C(x) $.
Thus, letting $ y := p^\rho_C(x) = \frac{\rho_C(x)}{|x|} x $, we have
$ \mathbf{D}p^\rho_C(x) = \frac{\rho_C(x)}{|x|} \mathbf{D}p^\rho
_C(y) $. Taking
$ y $ in place of $ x $ in \eqref{eq:Diffradp} and noting that $ \rho
_C(y) = |y| $,
we obtain
\[
\mathbf{D}p^\rho_C(y) v = v - \bigl\langle y - \llvert y
\rrvert \nabla\rho_C(y) , v \bigr\rangle \frac{y}{|y|^2} .
\]
Differentiating $ \rho_C(ky) = \rho_C(y) $ with respect to $ k $, we obtain
$ \langle\nabla\rho_C(y) , y \rangle = 0 $. Making use of
this identity, we find after
some calculation that $ \mathbf{D}p^\rho_C(y) $ is a projector.

If $ u $ is a unit outer normal vector at $ y $, then $ u $ is perpendicular
to the image of $ \mathbf{D}p^\rho_C(y) $. However, since $ \mathbf
{D}p^\rho_C(y) $ is
a projector, its image is the same as the set of its fixed points,
which are
precisely the vectors perpendicular to $ y - |y| \nabla\rho_C(y) $. Therefore,
$ u $ must be parallel to $ y - |y| \nabla\rho_C(y) $. Since $
\langle u , y \rangle > 0 $
and since $ \langle\nabla\rho_C(y) , y \rangle = 0 $,
we have $ u = \frac{y - |y| \nabla\rho_C(y)}{|y| \sqrt{1 + |\nabla
\rho_C(y)|^2}} $.
Thus, there is indeed a unique unit outer normal vector. Taking the
inner product
with $ y $, we find that $ |\nabla\rho_C(y)| = \tan\theta$.

Without loss of generality, we may assume that $ y/|y| $ is the first
base vector
and that $ \nabla\rho_C(y)/|\nabla\rho_C(y)| $ is the second one,
the latter
provided that $ \nabla\rho_C(y) \ne0 $. This way, we have
\[
\mathbf{D}p^\rho_C(y) = %
\begin{bmatrix} 0 &
\tan\theta&
\\
0 & 1 &
\\
& & \mathbf{I}_{d-2} \end{bmatrix} %
= %
\begin{bmatrix} \cos\theta& \sin\theta&
\\
- \sin\theta& \cos\theta&
\\
& & \mathbf{I}_{d-2} \end{bmatrix} %
\begin{bmatrix} 0 & 0 &
\\
0 & 1/\cos\theta&
\\
& & \mathbf{I}_{d-2} \end{bmatrix} %
\mathbf{I}_d
.
\]
The latter singular-value decomposition yields $ J_{d-1} p^\rho_C(y) =
1/\cos\theta$.
Recalling $ \mathbf{D}p^\rho_C(x) = \frac{\rho_C(x)}{|x|} \mathbf
{D}p^\rho_C(y) $, we obtain
(\ref{rad:J}).
\end{proof}

Before finally turning to the proof of Theorem~\ref{th:PerimBd}, we
still need some
inequalities regarding elementary and special functions. The first one
regards the
\emph{Mills ratio}:
\begin{equation}
\label{eq:Mills} R(x) := e^{x^2/2}
\int_x^\infty e^{- z^2/2} \,\mathrm{d}z =
\int _0^\infty e^{-tx - t^2/2} \,\mathrm{d}t .
\end{equation}
For $ y > 0 $, define
\begin{equation}
\label{eq:infMills} I(y) := \inf_{x \ge0} \bigl( x y + R(x) \bigr)
\end{equation}
and observe that $ I(y) > 0 $ and that $ I $ is strictly increasing.

\begin{lemma}
\label{lm:infMillsBd}
For all $ 0 < y < 1 $, the function $ I $ satisfies $ I(y) \ge2 \sqrt
{y(1 - y)} $.
\end{lemma}

\begin{proof}
By Formula~7.1.13 of Abramowitz and Stegun~\cite{AbS}, we have
$ R(x) \ge\frac{2}{x + \sqrt{x^2 + 4}} $
for all $ x \ge0 $. A straightforward calculation shows that the expression
$ \inf_{x \ge0} ( x y + \frac{2}{x + \sqrt{x^2 + 4}} ) $
equals $ 2 \sqrt{y(1 - y)} $ for $ y \le1/2 $ and $ 1 $ for $ y \ge
1/2 $.
%.% $ c = \sqrt{x^2 + 4} $ is very useful.
\end{proof}

\begin{lemma}
\label{lm:xalphaexp}
For all $ 0 \le x < \alpha$, we have
\begin{align}
\label{eq:xalphaexp-} \biggl( 1 - \frac{x}\alpha \biggr)^{- \alpha^2}
e^{- \alpha x} &\ge e^{x^2/2} ,
\\
\label{eq:xalphaexp+} \biggl( 1 - \frac{x}\alpha \biggr)^{\alpha^2 - 1}
e^{\alpha x} &\ge e^{- x^2/2} \biggl( 1 - \frac{x^3}{\alpha} \biggr) .
\end{align}
\end{lemma}

\begin{lemma}
\label{lm:GxalphapMin}
Consider the function
\[
G(x, \alpha, \beta) := \biggl( 1 - \frac{x}\alpha \biggr)^{- \alpha^2}
e^{- \alpha x} \biggl[ \beta+
\int_x^\alpha \biggl( 1 - \frac{y}\alpha
\biggr)^{\alpha^2 - 1} e^{\alpha y} \,\mathrm{d}y \biggr] .
\]
For all $ \alpha\ge1 $ and $ \beta\ge1/\sqrt e $, this function satisfies
\[
\inf_{x < \alpha} G(x, \alpha, \beta) = \inf_{0 \le x \le1}
G(x, \alpha, \beta) .
\]
\end{lemma}

The proofs of Lemmas~\ref{lm:xalphaexp} and \ref{lm:GxalphapMin} are
deferred to the
supplementary material \cite{BEJ1072-supp}.

\begin{proof}[Proof of Theorem~\ref{th:PerimBd}]
We basically follow Nazarov's~\cite{Naz} argument, tackling certain
technical matters
differently and expanding some arguments.
First, observe that if a convex set $ C $ has no interior, then it is contained
in the boundary of some half-space $ H $, so that $ \gamma(C) \le
\gamma(H) $. Therefore, in
the supremum in the definition of $ \gamma_d $, it suffices to consider
sets with non-empty interior. Next, if $ 0 \notin C $, we have $ \gamma
(C) \le
\gamma (C - p^\perp_C(0) ) $ (for details, see Section~4 of
Livshyts~\cite{Liv14}). Therefore, it suffices only to consider sets $
C $
with the origin in the closure and with non-empty interior. Moreover,
by continuity,
it suffices to take sets containing the origin in the interior.

Let $ C $ be a convex set with $ 0 \in\operatorname{Int}C $. Take a
random locally Lipschitzian map
$ G \colon\mathbb{R}^d \setminus\{ 0 \} \to\partial C $ with $
J_{d-1} G(x) \le J(x) $ for almost all
$ x \in A $, where $ J \colon\mathbb{R}^d \setminus\{ 0 \} \to(0,
\infty) $ is another random function
(random maps should be measurable as maps from the product of $ \mathbb
{R}^d \setminus\{ 0 \} $ and the
probability space with respect to the product of the Borel $ \sigma
$-algebra and the
$ \sigma$-algebra of the probability space). The random choices of $ G $ and $ J $ will depend on a
parameter $ p \in (0, 1] $ (see below). By Proposition~\ref
{pr:curviFubini}, we have
\begin{align*}
1 &=
\int_{\mathbb{R}^d} \phi_d(x) \,\mathrm{d}x \ge \mathbb{E}_p
\biggl[
\int_{\mathbb{R}^d \setminus\{ 0 \}} \phi_d(x) \frac{J_{d-1} G(x)}{J(x)}
\,\mathrm{d}x \biggr]
\\
&\ge
\int_{\partial C} \mathbb{E}_p \biggl[
\int_{G^{-1}(\{ y \})} \frac
{\phi_d}{J} \,\mathrm{d}\mathscr
H^1 \biggr] \mathscr H^{d-1}(\mathrm{d}y) .
\end{align*}
Thus,
\begin{equation}
\label{eq:PerimBd:recE} \gamma(C) \le\inf_{0 < p \le 1} \frac{1}{\inf_{y \in\partial C}
\xi_C(y, p)} ,
\end{equation}
where
\[
\xi_C(y, p) := \frac{1}{\phi_d(y)} \mathbb{E}_p \biggl[
\int _{G^{-1}(\{ y \})} \frac{\phi_d}{J} \,\mathrm{d}\mathscr
H^1 \biggr] .
\]
Now define $ G $ as follows: for $ x \in
\overline{C} $,
let $ G(x) := p^\rho_C(x) $; for $ x \in\mathbb{R}^d \setminus
\overline{C} $, let
$ G(x) := p^\rho_C(x) $ with probability $ 1 - p $ and $ G(x) :=
p^\perp_C(x) $
with probability $ p $. To define $ J(x) $, recall Lemma~\ref{lm:rad}
along with the fact that $ p^\perp_C $ is non-expansive. Thus, we may
take $ J(x) := 1 $ where
$ G = p^\perp_C $ and $ J(x) := ( \frac{\rho_C(x)}{|x|}
)^{d-1} \frac{1}{\cos\theta(p^\rho_C(x))} $
where $ G = p^\rho_C $; here, $ \theta(y) $ denotes the maximal angle
between $ y $ and
the outer normal of $ C $ at $ y $. Notice that the maximum is attained
because the set
of all unit outer normal vectors is compact, and is strictly less than
$ \pi/2 $
because $ 0 \in\operatorname{Int}C $; typically, the outer normal
vector is unique by Lemma~\ref{lm:rad}.
As a result, we have $ \xi_C(y, p) \ge\xi_{1,C}(y, p) + \xi
_{2,C}(y, p) $, where
\begin{align*}
\xi_{1,C}(y, p) :={}& \frac{\cos\theta(y)}{|y|^{d-1} \phi_d(y)} \biggl[
\int_{(p^\rho_C)^{-1}(\{ y \}) \cap\overline{C}} \llvert x \rrvert ^{d-1} \phi
_d(x) \mathscr H^1(\mathrm{d}x)
\\
& {}+ (1 - p)
\int_{(p^\rho_C)^{-1}(\{ y \}) \setminus\overline{C}} \llvert x \rrvert ^{d-1}
\phi_d(x) \mathscr H^1(\mathrm{d}x) \biggr] ,
\\
\xi_{2,C}(y, p):={}& \frac{p}{\phi_d(y)}
\int_{(p^\perp_C)^{-1}(\{
y \})} \phi_d(x) \mathscr H^1(
\,\mathrm{d}x) .
\end{align*}
Observe that $ \xi_{1,C}(y, p) = \cos\theta(y) \xi_1(|y|, d, p) $, where
\begin{equation}
\label{eq:xi1} %
\begin{aligned}[b] \xi_1(r, d, p):={}&
\frac{ e^{r^2/2}}{r^{d-1}} \biggl[
\int_0^r t^{d-1} e^{- t^2/2}
\,\mathrm{d}t + (1 - p)
\int_r^\infty t^{d-1} e^{- t^2/2}
\,\mathrm{d}t \biggr]
\\
% &=
% \frac{e^{r^2/2}}{r^{d-1}} \left[
% (1 - p) \int_0^\infty t^{d-1} e^{- t^2/2} \dl t + p \int_0^r t^{d-1}
%e^{- t^2/2} \dl t
% \right] =
% \\
={}&
\frac{e^{r^2/2}}{r^{d-1}} \biggl[
2^{d/2 - 1} (1 - p) \Gamma \biggl( \frac{d}{2} \biggr) + p \int_0^r
t^{d-1} e^{- t^2/2} \,\mathrm{d}t
\biggr] . \end{aligned} %
\end{equation}
As for $ \xi_{2,C}(y, p) $, observe that $ (p^\perp_C)^{-1}(\{ y \})
\supseteq
\{ y + su ;s > 0 \} $, where $ u $ is a unit outer normal vector at $ y $.
Take $ u $ with the maximal angle between $ u $ and $ y $. As a result,
we have
\[
\xi_{2,C}(y, p) \ge \frac{p}{\phi_d(y)}
\int_0^\infty\phi_d(y + tu) \,\mathrm{d}t =
p
\int_0^\infty e^{- t \langle y , u \rangle - t^2/2} \,\mathrm{d}t = p R
\bigl( \llvert y \rrvert \cos\theta(y) \bigr) ,
\]
recalling the Mills ratio defined in \eqref{eq:Mills}.
Combining all estimates after \eqref{eq:PerimBd:recE}, plugging into
the latter and
taking the supremum over all convex sets with the origin in the
interior, we find that
\begin{equation}
\label{eq:gammadBd1} \gamma_d \le\bar\gamma_d := \inf
_{0 < p \le1} \bar\gamma _{d,p} ,
\end{equation}
where
\[
\bar\gamma_{d,p} := \frac{1}{\inf_{r, c > 0} ( c \xi_1(r, d,
p) + p R(cr) )} .
\]
Substituting $ cr = b $ and recalling that the function $ I $ defined
in \eqref{eq:infMills}
is strictly increasing, we find the following alternative expression of
$ \bar\gamma_{d,p} $:
% provided that $ p > 0 $:
%
\begin{equation}
\label{eq:gammadp} \bar\gamma_{d,p} = \frac{1}{\inf_{r, b > 0} ( \frac{b}{r} \xi_1(r, d, p) + p
R(b) )} =
\frac{1}{p I ( \inf_{r > 0} \frac{1}{p r} \xi_1(r, d, p)
)} .
\end{equation}
For each $ d $, $ \bar\gamma_d $ can be evaluated numerically. Some
values are given in Table~\ref{tb:PerimBd}.
\begin{table}
\caption{Upper bounds on the Gaussian perimeter for some dimensions
(with all values rounded upwards)}
\label{tb:PerimBd}
\begin{tabular*}{\textwidth}{@{\extracolsep{\fill}}cc@{}}
\begin{tabular*}{0.45\textwidth}{@{\extracolsep{\fill}}@{}lll@{}}
\hline
$d$ & $\bar\gamma_d$ & $\bar\gamma_d/d^{1/4}$ \\
\hline
1 & 0.798 & 0.798 \\
2 & 0.864 & 0.726 \\
3 & 0.929 & 0.706 \\
4 & 0.981 & 0.694 \\
5 & 1.025 & 0.685 \\
6 & 1.063 & 0.679 \\
7 & 1.096 & 0.674 \\
8 & 1.126 & 0.670 \\ \hline
\end{tabular*}
&
\begin{tabular*}{0.45\textwidth}{@{\extracolsep{\fill}}@{}lll@{}}
\hline
$d$ & $\bar\gamma_d$ & $\bar\gamma_d/d^{1/4}$ \\
\hline
\hphantom{000}9 & 1.154 & 0.666 \\
\hphantom{00}10 & 1.179 & 0.663 \\
\hphantom{00}20 & 1.364 & 0.645 \\
\hphantom{00}50 & 1.666 & 0.627 \\
\hphantom{0}100 & 1.949 & 0.617 \\
\hphantom{0}200 & 2.288 & 0.609 \\
\hphantom{0}500 & 2.842 & 0.601 \\
1000 & 3.357 & 0.597 \\ \hline
\end{tabular*}
\end{tabular*}
\end{table}

\begin{remark}
For $ d = 1 $, we obtain the actual maximal Gaussian perimeter: we have
$ \gamma_1
= \bar\gamma_1 = \bar\gamma_{1,1} $. First, observe that
$ \inf_{r > 0} \frac{1}{r} \xi_1(r, 1, 1) = \inf_{r > 0} \frac
{e^{r^2/2}}{r} \int_0^r e^{- t^2/2} \,\mathrm{d}t = 1 $.
Differentiating \eqref{eq:Mills}, we find that $ R'(x) = - \int
_0^\infty t e^{- t x - t^2/2} \,\mathrm{d}t $ and
\newline
$ R''(x) = \int_0^\infty t^2 e^{- t x - t^2/2} \,\mathrm{d}t $. Since $
R'(0) = -1 $ and $ R''(x) > 0 $
for all $ x $, we have $ R'(x) \ge-1 $ for all $ x \ge0 $. Therefore,
for $ y = 1 $, the infimum
in \eqref{eq:infMills} is attained at $ x = 0 $, so that $ \bar\gamma
_{1,1} = 1/I(1) = 1/R(0)
= \sqrt{2/\pi} $.
\end{remark}

Now we continue with the estimation. From Stirling's formula with remainder
(e.g., Formula~6.1.38 of Abramowitz and Stegun~\cite{AbS}),
one can easily deduce
that $ \Gamma(x) \ge\sqrt{\frac{2 \pi}{x}} ( \frac{x}{e}
)^x $
for all $ x > 0 $. Plugging into \eqref{eq:xi1}, we obtain
\[
\frac{1}{p r} \xi_1(r, d, p) \ge \frac{e^{r^2/2}}{r^d} \biggl(
\frac{1 - p}{p} \sqrt{\frac{\pi}{d}} \biggl( \frac{d}{e}
\biggr)^{d/2} +
\int_0^r t^{d-1} e^{- t^2/2}
\,\mathrm{d}t \biggr) .
\]
Substituting $ \alpha:= \sqrt{d/2} $, $ x = \alpha- r^2/(2 \alpha)
$, $ y = \alpha- t^2/(2 \alpha) $,
we obtain after some calculation
\[
\inf_{r > 0} \frac{1}{p r} \xi_1(r, d, p) \ge
\frac{1}{2 \alpha} \inf_{x < \alpha} \biggl( 1 - \frac{x}{\alpha}
\biggr)^{- \alpha^2} e^{- \alpha x} \biggl[ \frac{1 - p}{p} \sqrt{2 \pi}
+
\int_x^\alpha \biggl( 1 - \frac{y}{\alpha}
\biggr)^{\alpha^2 - 1} e^{\alpha y} \,\mathrm{d}y \biggr] .
\]
Now suppose that $ \alpha\ge1 $ and $ \frac{1 - p}{p} \sqrt{2 \pi}
\ge\frac{1}{\sqrt{e}} $;
this is ensured if $ d \ge2 $ and $ p < 0.8 $. In this case, we can
apply Lemma~\ref{lm:GxalphapMin}
to reduce the infimum over $ x < \alpha$ to the infimum over $ [0, 1] $.
By Lemma~\ref{lm:xalphaexp}, we can further estimate
\[
\inf_{r > 0} \frac{1}{p r} \xi_1(r, d, p) \ge
\frac{1}{2 \alpha} \inf_{0 \le x \le1} e^{x^2/2} \biggl[
\frac{1 - p}{p} \sqrt{2 \pi} +
\int_x^\alpha e^{- y^2/2} \biggl( 1 -
\frac{y^3}{\alpha} \biggr) \,\mathrm{d}y \biggr] .
\]
Since $ \alpha\ge1 $, $ y \ge\alpha$ implies $ y^3 \ge\alpha$, so
that the upper
limit $ \alpha$ can be replaced with the infinity:
\begin{align*}
\inf_{r > 0} \frac{1}{p r} \xi_1(r, d, p) &
\ge \frac{1}{2 \alpha} \inf_{0 \le x \le1} e^{x^2/2} \biggl[
\frac{1 - p}{p} \sqrt{2 \pi} +
\int_x^\infty e^{- y^2/2} \biggl( 1 -
\frac{y^3}{\alpha} \biggr) \,\mathrm{d}y \biggr]
\\
&= \frac{1}{2 \alpha} \inf_{0 \le x \le1} \biggl[ \frac{1 - p}{p}
\sqrt{2 \pi} e^{x^2/2} + R(x) - \frac{x^2 +
2}{\alpha} \biggr]
\\
&\ge \frac{1}{2 \alpha} K(p) - \frac{3}{2 \alpha^2}
\\
&= \frac{1}{\sqrt{2d}} K(p) - \frac{3}{d} ,
\end{align*}
where
$
K(p) := \inf_{0 \le x \le1} [
\frac{1 - p}{p} \sqrt{2 \pi} e^{x^2/2} + R(x)
]
$.
Plugging into \eqref{eq:gammadp} and applying
\newline
Lemma~\ref{lm:infMillsBd}, we find that
\[
\bar\gamma_{d,p} \le \frac{1}{2 p \sqrt{ (
\frac{1}{\sqrt{2d}} K(p) - \frac{3}{d} ) ( 1 - \frac
{1}{\sqrt{2d}} K(p) + \frac{3}{d} )
}} ,
\]
provided that $ d \ge2 $, $ p \le0.8 $ and
$ d > \max \{ \frac{(K(p))^2}{2}, \frac{18}{(K(p))^2} \} $.
As $ d \to\infty$, the preceding upper bound asymptotically equals
$
\frac{d^{1/4}}{2^{3/4} p \sqrt{K(p)}}
$.
Now choose $ p $ so that this asymptotic bound is optimal, that is, so
that $ p^2 K(p) $ is
maximal. Numerical calculation shows that this occurs approximately at
$ p = p^* := 0.72 $
(which is less than $ 0.8 $). Moreover, one can numerically check that
$ K(p^*) > K^* := 1.98 $.
This indicates that the coefficient at $ d^{1/4} $ in the bound on $
\gamma_d $ can be set to
$
\frac{1}{2^{3/4} p^* \sqrt{K^*}} < 0.59
$.

Choosing $ p = p^* $, we re-estimate $ \bar\gamma_d $, using
Lemma~\ref{lm:infMillsBd}
once again:
\begin{align*}
\bar\gamma_d \le\bar\gamma_{d,p^*} &\le \frac{1}{p^* I ( \frac{K^*}{\sqrt{2 d}} - \frac{3}{d}
)}
\\
&\le \frac{1}{2 p^* \sqrt{
( \frac{K^*}{\sqrt{2d}} - \frac{3}{d} ) ( 1 -
\frac{K^*}{\sqrt{2d}} + \frac{3}{d} )
}} =
\\
&\le \frac{0.59 d^{1/4}}{\sqrt{1 - (
\frac{3 \sqrt{2}}{K^*} + \frac{K^*}{\sqrt{2}}
) d^{-1/2}}}
\\
&\le \frac{0.59 d^{1/4}}{\sqrt{1 - 3.55 d^{-1/2}}} ,
\end{align*}
provided that $ d \ge3.55^2 $, that is, $ d \ge13 $. Taking $ d = 932
$, observe that
\newline
$ ( 1 - 3.55 d^{-1/2} )^{-1/2} \le1 + ( \frac
{1}{0.59} \sqrt{\frac{2}{\pi}} - 1 )
d^{-1/4} $; this inequality also holds in the limit as $ d \to\infty
$. Since the function
$ x \mapsto ( 1 - 3.55 x^2 )^{-1/2} $ is convex, the latter
inequality must hold for all
$ d \ge932 $. This completes the proof for the latter case. For $ d <
932 $, the desired result can be
verified numerically, evaluating \eqref{eq:gammadBd1} directly.
\end{proof}

\section*{Acknowledgements}

The author is grateful to Mihael Perman for a fruitful discussion
that led to the appearance of this paper, and for useful comments.%

%%%%%%%%%%%%%%%%%%%%%%%%%%%%%%%%%%%%%%%%%%%%%%%%%%%%%%%%%%%%%%%%%%%%%%%%%%
%\begin{appendix}
%%\section{}
%\end{appendix}

% Change title according to manuscript
%\section*{Acknowledgements}%% Put support info here

%\begin{supplement}
%\stitle{}
%%%%%%%%%%%%%%%%%%%%%%%%%%%%%%%%%%%%%
%%\slink[url]{} %[url,text={...}] %use to break url
%%%%%%%%%%%%%%%%%%%%%%%%%%%%%%%%%%%%%
%% OR
%%%%%%%%%%%%%%%%%%%%%%%%%%%%%%%%%%%%%
%%\slink[doi]{10.3150/00-BEJXXXXSUPP} %[doi,text={...}] %use to break
%doi
%%\sdatatype{.?} %% pdf,doc,zip
%%\sfilename{BEJ000supp.?} %% pdf,doc,zip
%%%%%%%%%%%%%%%%%%%%%%%%%%%%%%%%%%%%%
%\sdescription{}
%\end{supplement}

\begin{supplement}
\stitle{Proofs of certain technical issues}
\slink[doi]{10.3150/18-BEJ1072SUPP}
\sdatatype{.pdf}
\sfilename{bej1072supp.pdf}
\sdescription{The supplementary file contains a proof of continuous
differentiability of $ f_A^{\varepsilon} $
in Lemma~\ref{lm:Smoothdelta}, and proofs of Lemmas~\ref
{lm:xalphaexp} and \ref{lm:GxalphapMin}.}
\end{supplement}

% imsref loaded by dianan, 2018-11-21 14:11:36
%

% \printhistory
\end{document}